\pdfoutput=1

\documentclass[12pt]{amsart}
\usepackage[margin=1in]{geometry} 

\usepackage{amsmath,amssymb,amsthm,bm,colonequals,graphicx,mathrsfs,mathtools,microtype,stmaryrd}
\usepackage[shortlabels]{enumitem}
\usepackage[colorinlistoftodos]{todonotes}

\numberwithin{equation}{section}

\theoremstyle{plain}

\newtheorem{theorem}{Theorem}[section] 

\newtheorem{lemma}[theorem]{Lemma} 

\newtheorem{proposition}[theorem]{Proposition} 

\newtheorem{proposition-definition}[theorem]{Proposition-Definition} 

\theoremstyle{definition}

\newtheorem{definition}[theorem]{Definition} 

\theoremstyle{remark}

\newtheorem{remark}[theorem]{Remark} 

\newtheorem{observation}[theorem]{Observation} 

\newcommand{\Aff}{\mathbb{A}}

\newcommand{\EE}{\mathbb{E}}
\newcommand{\FF}{\mathbb{F}}

\newcommand{\PP}{\mathbb{P}}
\newcommand{\QQ}{\mathbb{Q}}
\newcommand{\RR}{\mathbb{R}}

\newcommand{\ZZ}{\mathbb{Z}}

\newcommand{\lcm}{\operatorname{lcm}}

\newcommand{\vol}{\operatorname{vol}}

\newcommand{\abs}[1]{\lvert #1 \rvert}
\newcommand{\card}[1]{\lvert #1 \rvert}
\newcommand{\norm}[1]{\lVert #1 \rVert}
\newcommand{\floor}[1]{\lfloor #1 \rfloor}

\newcommand{\eps}{\epsilon}

\newcommand{\rad}{\operatorname{rad}}
\newcommand{\Supp}{\operatorname{Supp}}

\newcommand{\map}{\operatorname}
\newcommand{\mscr}{\mathscr}
\newcommand{\mcal}{\mathcal}
\newcommand{\mf}{\mathfrak}

\newcommand{\defeq}{\colonequals}

\newcommand{\maps}{\colon}

\newcommand{\belongs}{\subseteq}

\newcommand{\set}[1]{\{#1\}}

\newcommand{\grad}{\nabla}

\usepackage[pagebackref=true]{hyperref}
\usepackage[alphabetic,initials]{amsrefs}

\title{Prime Hasse principles via Diophantine second moments}
\date{} 
\author{Victor Y. Wang}
\address{Fine Hall, 304 Washington Road, Princeton, NJ 08540, USA}
\address{Courant Institute, 251 Mercer Street, New York, NY 10012, USA}
\address{IST Austria, Am Campus 1, 3400 Klosterneuburg, Austria}
\email{vywang@alum.mit.edu}

\begin{document}

\begin{abstract}
We show that almost all primes $p\not\equiv \pm 4 \bmod{9}$ are sums of three cubes,
assuming a conjecture due to Hooley, Manin, et al.~on cubic fourfolds.
This conjecture is approachable under standard statistical hypotheses on geometric families of $L$-functions.

\end{abstract}

\maketitle

\setcounter{tocdepth}{3}






\section{Introduction}

Let $F_0(\bm{y})=F_0(y_1,y_2,y_3)\defeq y_1^3+y_2^3+y_3^3$
and $F_0(S)\defeq \set{F_0(\bm{y}): \bm{y}\in S}$.
For well-known local reasons, $F_0(\ZZ^3)\belongs \set{a\in \ZZ: a\not\equiv\pm4\bmod{9}}$.
The \emph{Hasse principle}
holds if the set
\begin{equation}
\label{EQN:define-Hasse-exceptional-set}
\mcal{E} \defeq \set{a\in \ZZ: a\not\equiv\pm4\bmod{9}} \setminus F_0(\ZZ^3)
\end{equation}
is empty.
The analog of $\mcal{E}$ for $5y_1^3+12y_2^3+9y_3^3$ contains a sparse sequence
\cite{ghosh2017integral}*{p.~691, footnote~3},\footnote{See \cite{lyczak2023cubic} for a general study of Brauer--Manin obstructions for ternary diagonal cubic forms.}
produced by means that do not apply to $F_0$ \cite{colliot2012groupe}*{p.~1304}.
We confine ourselves to a statistical analysis of $\mcal{E}$ relative to $\ZZ$ and the set of primes.
For ``critical'' equations such as $F_0(\bm{y})=a$, ``subcritical'' statistical frameworks (such as those of \cites{vaughan1980ternary,brudern1991ternary,hooley2016representation}) break down, and new features come into play \cites{ghosh2017integral,diaconu2019admissible}.

Here ``critical'' refers to the
well-known fact that
if, say,
$X\ge 1$ and $A\in [X^2, X^3]$, then
\begin{equation*}
\EE_{\abs{a}\le A}[\#\set{\bm{y}\in \ZZ^3: \norm{\bm{y}}\le X,\; F_0(\bm{y})=a}]
\ll \log(1 + X/A^{1/3}).
\end{equation*}
(Any solutions only ``barely'' exist!)
To produce integral solutions to $F_0(\bm{y})=a$ in general, one must take $X/A^{1/3}\to \infty$, not fixed.
In fact,
for any fixed $\lambda\ge 1$, the set
\begin{equation*}
\set{a\not\equiv\pm4\bmod{9}: a\notin F_0([-\abs{a}^{1/3}\lambda, \abs{a}^{1/3}\lambda]^3)}
\belongs \ZZ
\end{equation*}
has lower density $>0$ \cite{diaconu2019admissible}*{\S1}.
For further discussion, see \cite{heath1992density}*{pp.~622--623}.

The scarcity of solutions also forces us to pass from the \emph{sparse} setting of $F_0$ to a \emph{richer} setting in $6$ variables.
If $r_3(a)\defeq \#\set{\bm{y}\in \ZZ_{\geq 0}^3: F_0(\bm{y}) = a}$, then by Cauchy,
\begin{equation}
\label{INEQ:basic-Cauchy}
\card{F_0(\ZZ_{\geq 0}^3) \cap [0,A]}
\ge (\EE_{0\le a\le A}[r_3(a)])^2 / \EE_{0\le a\le A}[r_3(a)^2].
\end{equation}
Whereas $\EE_{0\le a\le A}[r_3(a)] \asymp 1$ by classical geometry of numbers, the mean square $\EE_{0\le a\le A}[r_3(a)^2]$ corresponds to a difficult point count in $3+3=6$ cubes, which we now introduce.

Write $\bm{y}=(y_1,y_2,y_3)$, $\bm{z}=(z_1,z_2,z_3)$,
and let $\bm{x}=(x_1,\dots,x_6)\defeq (y_1,y_2,y_3,z_1,z_2,z_3)$.
Let
\begin{equation}
\label{EQN:define-fiber-product-form-F-from-F_0}
F(\bm{x})=F(\bm{y},\bm{z})\defeq F_0(\bm{y})-F_0(\bm{z}).
\end{equation}
Let $\Upsilon$ denote the set of $3$-dimensional vector spaces $L\belongs \QQ^6$ over $\QQ$ such that $F\vert_L = 0$.
(The equation $\bm{y} = \bm{z}$ cuts out one such $L$.
All other $L\in \Upsilon$ can be generated from $\bm{y} = \bm{z}$ by suitable permutations and negations of variables.)
Call a tuple $\bm{x}\in \ZZ^6$ \emph{special} if
\begin{equation}
\label{COND:x-is-in-the-union-of-special-linear-subspaces}
\bm{x}\in \bigcup_{L\in \Upsilon} L.
\end{equation}
We will concentrate on smoothly weighted point counts away from the origin $\bm{0}$; there are, however, standard methods to pass to unrestricted, unweighted point counts.

Given integers $X,d,q\geq 1$, and a function $w\in C^\infty_c(\RR^6)$ supported away from $\bm{0}$, let
\begin{align}
N_w(X;d) &\defeq \sum_{\bm{x}\in \ZZ^6:\,
F(\bm{x})=0,\; d\mid F_0(\bm{y}),F_0(\bm{z})} w(\bm{x}/X),
\label{EQN:define-integral-point-count-N_w(X;d)} \\
\rho(q;d) &\defeq q^{-5} \cdot
\#\set{\bm{x}\in (\ZZ/q\ZZ)^6: q\mid F(\bm{x}) \text{ and }
\gcd(q,d) \mid F_0(\bm{y}), F_0(\bm{z})},
\label{EQN:define-mod-q-density-rho(q;d)} \\
\sigma_p(d) &\defeq \lim_{l\to\infty}{\rho(p^l;d)},
\quad \sigma_{\infty,w} \defeq
\lim_{\eps\to 0}{(2\eps)^{-1}\int_{\abs{F(\bm{x})}\leq\eps} w(\bm{x})}\, d\bm{x},
\label{EQN:define-local-densities-sigma_v} \\
E_w(X;d) &\defeq N_w(X;d)
- \mf{S}(d)\cdot \sigma_{\infty,w}\cdot X^3
- \sum_{\textnormal{special }\bm{x}\in \ZZ^6:\,
d\mid F_0(\bm{y}),F_0(\bm{z})} w(\bm{x}/X)
\label{EQN:define-HLH-error-E_w(X;d)},
\end{align}
where $\mf{S}(d) \defeq \prod_{\textnormal{$p$ prime}} \sigma_p(d)$ is the \emph{singular series} associated to
the equation $F(\bm{x})=0$ with the congruence condition $F_0(\bm{y})\equiv F_0(\bm{z})\equiv 0\bmod{d}$.
By directly generalizing \cite{hooley1986some}*{Conjecture~2 for $l=3$}, one conjectures
\begin{equation}
\label{EQN:soft-HLH-general-homogeneous-weight}
\lim_{X\to \infty} X^{-3} E_w(X;d) = 0.
\end{equation}
This is a Manin-type randomness-structure dichotomy conjecture on the $6$-variable cubic hypersurface $F=0$
(based on \emph{special subvarieties}, i.e.~thin sets of type~I),
compatible with the general framework of \cite{franke1989rational}, \cite{vaughan1995certain}*{Appendix}, \cite{peyre1995hauteurs}, et al.

By \eqref{INEQ:basic-Cauchy},
$F_0(\ZZ_{\geq 0}^3)$ has lower density $>0$,
assuming $E_w(X;1)\ll X^3$ holds for a suitable $w$.
There are conjectures on the density of $F_0(\ZZ_{\geq 0}^3)$ \cite{deshouillers2006density}, but we focus on $F_0(\ZZ^3)$.
Diaconu, building on \cite{ghosh2017integral}, showed that \eqref{EQN:define-Hasse-exceptional-set} has density $0$,
assuming a version of \eqref{EQN:soft-HLH-general-homogeneous-weight} for $d=1$ where $N_w(X;1)$ is replaced by a point count over a fairly skew region \cite{diaconu2019admissible}*{$R^\ast_N$ on p.~24}.
In fact, as we will see, \eqref{EQN:soft-HLH-general-homogeneous-weight} itself suffices.
We can say even more about \eqref{EQN:define-Hasse-exceptional-set}, assuming that there exists a triple $(\xi,\delta,k)\in \set{0,1}\times \RR_{>0}\times \ZZ_{\geq 1}$ for which the following holds:
\begin{equation}
\label{EQN:hard-HLH-clean-weight-level-d}
E_w(X;d) \ll \norm{w}_{k,\infty} B(w)^k X^{3-\delta},
\quad\textnormal{uniformly over $d\leq X^{\xi\delta}$ and clean $w\in C^\infty_c(\RR^6)$}.
\end{equation}

We define the quantities $\norm{w}_{k,\infty}$, $B(w)$ in \S\ref{SUBSEC:conventions}; they measure the complexity of $w$.
We call a function $f\maps\RR^s\to\RR$ \emph{clean} if it is supported away from the coordinate hyperplanes,
i.e.~
\begin{equation}
\label{EQN:condition-for-clean}
(\Supp{f}) \cap \set{\bm{u}\in \RR^s: u_1\cdots u_s = 0} = \emptyset.
\end{equation}

\begin{theorem}
\label{THM:main-result-on-integer-values}
Suppose that \eqref{EQN:soft-HLH-general-homogeneous-weight} for $d=1$ holds for all clean functions $w\in C^\infty_c(\RR^6)$.
Then \eqref{EQN:define-Hasse-exceptional-set} has density $0$ in $\ZZ$.
Now fix $(\delta, k)\in \RR_{>0}\times \ZZ_{\geq 1}$, and assume \eqref{EQN:hard-HLH-clean-weight-level-d} for $\xi=0$.
Then $\card{\mcal{E}\cap [-A,A]}\ll_\eps A/(\log{A})^{1-\eps}$ for all integers $A\geq 2$ (for all $\eps>0$).
\end{theorem}

The qualitative density result in Theorem~\ref{THM:main-result-on-integer-values} was proven in \cite{wang2022thesis}*{Theorem~2.1.8} using the strategy of \cite{diaconu2019admissible}*{\S\S2--3}.
But our present approach is stronger and more flexible.

\begin{theorem}
\label{THM:main-result-on-prime-values}
Fix $(\delta, k)\in \RR_{>0}\times \ZZ_{\geq 1}$.
Assume \eqref{EQN:hard-HLH-clean-weight-level-d} for $\xi=1$.
Then \eqref{EQN:define-Hasse-exceptional-set} contains at most $O_\eps(A/(\log{A})^{2-\eps})$ primes $p\leq A$, for any integer $A\geq 2$ (for any $\eps>0$).
\end{theorem}

Before proceeding, some notes on \eqref{EQN:soft-HLH-general-homogeneous-weight}, \eqref{EQN:hard-HLH-clean-weight-level-d} are in order:
\begin{enumerate}
\item Unconditionally, $E_w(X;1)\ll_{w,\eps} X^{7/2} (\log{X})^{\eps-5/2}$ for $X\geq 2$ \cite{vaughan2020some}.
Improving on this is difficult; classically, \cite{hua1938waring} proved $E_w(X;1)\ll_{w,\eps} X^{7/2+\eps}$.

\item Conditionally, \cites{hooley1986HasseWeil,hooley_greaves_harman_huxley_1997,heath1998circle} proved the near-optimal bound $E_w(X;1)\ll_{w,\eps} X^{3+\eps}$ assuming GRH for geometric $L$-functions.

\item Building thereon, \cite{wang2023_HLH_vs_RMT} proved
\eqref{EQN:soft-HLH-general-homogeneous-weight} for $d=1$
(and in fact $E_w(X;1)\ll_w X^{3-\delta}$ for some $\delta$)
for all clean functions $w\in C^\infty_c(\RR^6)$,
assuming mainly some statistical hypotheses of Random Matrix Theory type on geometric families of $L$-functions.

\item The world of $L$-functions is very rich.
Over function fields, for instance, GRH is known and \cites{bergstrom2023hyperelliptic,MPPRW,ratio2024forthcoming} give a promising route towards \cite{wang2023_HLH_vs_RMT}'s hypotheses.
See \cite{BGW2024forthcoming} for further details.

\item The hypothesis \eqref{EQN:hard-HLH-clean-weight-level-d} is essentially a version of \eqref{EQN:soft-HLH-general-homogeneous-weight} with a specified power saving, uniform over a specified range of integers $d$ and weights $w$.
The conditional approach of \cite{wang2023_HLH_vs_RMT} should in principle lead to \eqref{EQN:hard-HLH-clean-weight-level-d} for some $(\delta, k)\in \RR_{>0}\times \ZZ_{\geq 1}$.

\item If $E_w(X;1)\ll_w X^\theta$ for all $w$, where $\theta\ge 3$, then
$E_w(X;d)\ll d^{O(1)} \norm{w}_{0,\infty} (B(w) X)^\theta$.
The uniformity in \eqref{EQN:hard-HLH-clean-weight-level-d} is thus natural,
but it requires extra work if $\theta<3$.
\end{enumerate}

How are Theorems~\ref{THM:main-result-on-integer-values}--\ref{THM:main-result-on-prime-values} proven?
Given a compactly supported function $\nu\maps \RR^3\to \RR$, let
\begin{equation}
\label{EQN:define-integral-point-count-N_a,nu(X)}
N_{a,\nu}(X) \defeq \sum_{\bm{y}\in \ZZ^3:\, F_0(\bm{y})=a} \nu(\bm{y}/X).
\end{equation}
To prove Theorem~\ref{THM:main-result-on-integer-values},
we will choose weights $\nu$ (smooth as in \cite{hooley2016representation}, yet tentacled as in \cites{ghosh2017integral,diaconu2019admissible}),
and bound $N_{a,\nu}(X)$ in approximate variance over $a\ll X^3$.
To prove Theorem~\ref{THM:main-result-on-prime-values}, we will also need precise estimates (not just bounds) for such variances, and not just over $\ZZ$ but also over arithmetic progressions $a\equiv 0\bmod{d}$ with $d\leq X^\delta$.

In \S\ref{SEC:general-variance-setup-and-estimation}, we will define and estimate an approximate variance.
In \S\ref{SEC:statistics-of-s_a(K)}, we will show that certain truncated singular series are typically sizable.
In \S\ref{SEC:apply-variance-estimates}, we will apply our estimates from \S\S\ref{SEC:general-variance-setup-and-estimation}--\ref{SEC:statistics-of-s_a(K)} to prove Theorems~\ref{THM:main-result-on-integer-values} and~\ref{THM:main-result-on-prime-values}.
Here it is important to allow $\nu$ to deform with $X$.
For fixed $\nu$, counting prime values of $F_0$ (with or without multiplicity) up to constant factors remains a challenge, even conditionally; see \cite{destagnol2019rational}*{Conjecture~A.3} (Bateman--Horn in several variables) and the present \S\ref{SEC:nonnegative-cubes} for some concrete open questions in this direction.
(It might also be interesting to ask analogous questions for sequences other than the primes.)

How many solutions to $x^3+y^3+z^3=a$ of a given size can our methods produce, for typical $a$?
By slightly modifying \S\ref{SEC:apply-variance-estimates},
one could provide several kinds of answers to this question.
A comprehensive discussion would be tedious, so we limit ourselves to a murky remark:

\begin{remark}
In Theorems~\ref{THM:main-result-on-integer-values} and~\ref{THM:main-result-on-prime-values}, there are three levels of hypotheses in total, say H1, H2, H3, from weakest to strongest.
H1 implies that if $h(a)\to \infty$ as $\abs{a}\to \infty$, then there exists $f\maps \ZZ \to \RR$, with $\lim_{\abs{a}\to\infty} f(a) = \infty$, such that for almost all integers $a\not\equiv\pm4\bmod{9}$, the equation $x^3+y^3+z^3=a$ has $\geq f(a)$ solutions with $\max(\abs{x},\abs{y},\abs{z})\le h(a)\cdot \abs{a}^{1/3}$.\footnote{One may simultaneously satisfy the condition $\max(\abs{x},\abs{y},\abs{z}) \ge f(a) \min(\abs{x},\abs{y},\abs{z})$, as the referee has suggested.
This is a weak analog of the ``minicube'' problem discussed in \cite{brudern2010waring}.
Minicubes do sometimes, if rarely, occur in practice; see the lopsided solution of \cite{booker2021question} to $x^3+y^3+z^3 = 3$, for instance.}
Under H2 (resp.~H3), one can show that if $g(a)\to 0$ as $\abs{a}\to \infty$, then $x^3+y^3+z^3=a$ has $\geq g(a)\cdot \log(1+\abs{a})$ solutions for almost all integers (resp.~primes) $a\not\equiv\pm4\bmod{9}$.
\end{remark}

We expect that Theorems~\ref{THM:main-result-on-integer-values} and~\ref{THM:main-result-on-prime-values} can be generalized from $F_0$ to arbitrary ternary cubic forms $G_0$ with nonzero discriminant,
if one modifies \eqref{EQN:condition-for-clean} and \eqref{EQN:condition-for-very-clean}.
For \eqref{EQN:condition-for-clean}, see \cite{wang2022thesis}*{Definition~1.4.3}.
For \eqref{EQN:condition-for-very-clean}, one needs that for any $L$ in the (finite) set $\Upsilon$,
if $I(L)$ denotes the ideal of $\QQ[\bm{y},\bm{z}]$ defining $L$, then $I(L)\cap \QQ[\bm{y}]$ is a principal ideal of $\QQ[\bm{y}]$.

The proof of Theorem~\ref{THM:main-result-on-prime-values} might also adapt to the Markoff cubic $x^2+y^2+z^2-xyz$ to \emph{unconditionally} extend \cite{ghosh2017integral}*{Theorem~1.2(ii)} to primes,
provided one can handle certain quadratic subtleties uncovered in \cite{ghosh2017integral}.
For practical reasons, we focus on $x^3+y^3+z^3$.

\subsection{Conventions}
\label{SUBSEC:conventions}


We let $\ZZ_{\geq 0}\defeq \set{n\in \ZZ: n\geq 0}$, and define $\ZZ_{\geq 1}, \RR_{>0}, \ldots$ similarly.
For a finite nonempty set $S$, we let $\EE_{b\in S}[f(b)]\defeq \card{S}^{-1} \sum_{b\in S} f(b)$.
For an event $E$, we let $\bm{1}_E\defeq 1$ if $E$ holds, and $\bm{1}_E\defeq 0$ otherwise.
We let $e(t)\defeq e^{2\pi it}$,
and $e_r(t)\defeq e(t/r)$.

For a vector $\bm{u}\in \RR^s$, we let $\norm{\bm{u}}\defeq \max_{i}(\abs{u_i})$ and $d\bm{u}\defeq du_1\cdots du_s$.
We let $C^\infty_c(\RR^s)$ denote the set of smooth compactly supported functions $\RR^s\to \RR$.
Given a function $f\maps\RR^s\to\RR$, we let $\Supp{f}$ denote the closure of $\set{\bm{u}\in \RR^s: f(\bm{u})\neq 0}$ in $\RR^s$.
We let $\partial_{x}\defeq \partial/\partial x$.
We write $\int_X dx\,f(x)$ to mean $\int_X f(x)\,dx$;
we write $\int_{X\times Y} dx\,dy\,f(x,y)$ to mean $\int_X dx\,(\int_Y dy\,f(x,y))$.

Given $f=f(u_1,\dots,u_s)\in C^\infty_c(\RR^s)$ and $k\in \ZZ_{\geq 0}$, we define the Sobolev norm
\begin{equation}
\label{EQN:define-Sobolev-norm}
\norm{f}_{k,\infty}\defeq \max_{\substack{\alpha_1,\dots,\alpha_s\in \ZZ_{\geq 0}: \\ \alpha_1+\dots+\alpha_s\leq k}}\,
\max_{\RR^s}{\abs{\partial_{u_1}^{\alpha_1} \cdots \partial_{u_s}^{\alpha_s} f}}.
\end{equation}
Given a clean, compactly supported function $f\maps\RR^s\to\RR$,
we let $B(f)$ denote the smallest integer $B\geq 1$ such that $B^{-1}\leq \abs{u_1},\dots,\abs{u_s}\leq B$ for all $\bm{u}\in \Supp f$.

We write $f\ll_S g$, or $g\gg_S f$, to mean $\abs{f} \leq Cg$ for some $C = C(S)>0$.
We let $O_S(g)$ denote a quantity that is $\ll_S g$.
We write $f\asymp_S g$ if $f\ll_S g\ll_S f$.

We let $v_p(-)$ denote the usual $p$-adic valuation.
For integers $n\geq 1$, we let $\phi(n)$ denote the totient function, $\tau(n)$ the divisor function, $\mu(n)$ the M\"{o}bius function, $\omega(n)$ the number of distinct prime factors of $n$, and $\rad(n)$ the radical of $n$.

\section{General variance setup and estimation}
\label{SEC:general-variance-setup-and-estimation}

\newcommand{\emm}{n}
\newcommand{\enn}{m}

Let $F_a\defeq F_0 - a$ and $\nu\in C^\infty_c(\RR^3)$.
Suppose $\bm{0}\notin \Supp\nu$.
Let us define local densities corresponding to the point count \eqref{EQN:define-integral-point-count-N_a,nu(X)}.
For all $a\in \ZZ \setminus \set{0}$, let
\begin{equation}
\label{EQN:define-F_a-p-adic-density}
\sigma_{p,a}
\defeq \lim_{l\to\infty}
(p^{-2l}\cdot\#\set{\bm{y}\in (\ZZ/p^l\ZZ)^3: F_a(\bm{y})=0});
\end{equation}
this exists by \cite{chambert2010igusa}*{\S5.4} since $F_a=0$ is smooth in $\Aff^3_\QQ$.
Similarly, but for all $a\in\RR$, let
\begin{equation}
\label{EQN:define-F_a-real-density}
\sigma_{\infty,a,\nu}(X)
\defeq \lim_{\eps\to 0}{(2\eps)^{-1}
\int_{\bm{y}\in \RR^3:\, \abs{F_a(\bm{y})}\le \eps} d\bm{y}\, \nu(\bm{y}/X)};
\end{equation}
this exists because the surface $F_a=0$ in $\RR^3\setminus \set{\bm{0}}$ is smooth (even if $a=0$).
At least for cube-free $a\ne 0$, the density $\sigma_{p,a}$ and a real density are computed on \cite{heath1992density}*{p.~622}.

For $a\in\ZZ\setminus \set{0}$,
informally write $\mathfrak{S}_a
\defeq\prod_{\textnormal{$p$ prime}}\sigma_{p,a}$,
and consider the \emph{Hardy--Littlewood prediction} $\mathfrak{S}_a\cdot\sigma_{\infty,a,\nu}(X)$ for $N_{a,\nu}(X)$.
Smaller moduli should have a greater effect in $\mf{S}_a$;
furthermore,
$\mf{S}_a$ itself---\emph{as is}---can be subtle (see \S\ref{SEC:statistics-of-s_a(K)}).
So in \eqref{EQN:define-approximate-variance} below,
we use a ``restricted'' version of $\mathfrak{S}_a\cdot\sigma_{\infty,a,\nu}(X)$.
For technical quantitative reasons, we use a series approximation \eqref{EQN:define-s_a(K)} to $\mathfrak{S}_a$, rather than a product approximation as in \cites{ghosh2017integral, diaconu2019admissible, wang2022thesis}.

\begin{definition}
For integers $\enn\geq 1$, let
\begin{equation}
\label{EQN:define-T_a(m)}
T_a(\enn)
\defeq \sum_{u\in (\ZZ/\enn\ZZ)^\times} \sum_{\bm{y}\in (\ZZ/\enn\ZZ)^3} e_\enn(u F_a(\bm{y}))
= \sum_{\substack{1\le u\le \enn: \\ \gcd(u,\enn)=1}}\,
\sum_{1\le y_1,y_2,y_3\le \enn} e_\enn(u F_a(y_1,y_2,y_3))
\end{equation}
whenever $a$ lies in $\ZZ$ or $\ZZ/\enn\ZZ$ (or maps canonically into $\ZZ/\enn\ZZ$).
For $a\in \ZZ$ and $K\in \ZZ_{\ge 1}$, let
\begin{equation}
\label{EQN:define-s_a(K)}
s_a(K)\defeq \sum_{\enn\leq K} \enn^{-3} T_a(\enn).
\end{equation}
For integers $d\geq 1$, let $d\ZZ\defeq \set{a\in \ZZ: d\mid a}$ and define the \emph{$K$-approximate variance}
\begin{equation}
\label{EQN:define-approximate-variance}
\map{Var}(X,K;d)
\defeq \sum_{a\in d\ZZ}[N_{a,\nu}(X) - s_a(K)\sigma_{\infty,a,\nu}(X)]^2.
\end{equation}
\end{definition}

Both $\nu$ and $T_a(\enn)$ are real-valued, so \eqref{EQN:define-approximate-variance} is a reasonable definition.
For later use, let
\begin{align}
\nu^{\otimes 2}(\bm{x})=\nu^{\otimes 2}(\bm{y},\bm{z})&\defeq \nu(\bm{y})\nu(\bm{z}),
\label{EQN:define-symmetric-separable-w-weight-from-nu} \\
\norm{\nu}_{L^2(\RR^3)}^2 &\defeq \int_{\bm{y}\in \RR^3} d\bm{y}\,\nu(\bm{y})^2.
\label{EQN:define-L2-norm-of-nu}
\end{align}
For clean $\nu$, we will
rewrite $\map{Var}(X,K;d)$ after expanding the square
(Theorem~\ref{THM:unconditional-variance-evalulation-over-integers} below).
Due to the nonnegativity of squares in \eqref{EQN:define-approximate-variance}, the Selberg sieve can then be used to bound a variant of \eqref{EQN:define-approximate-variance} over primes, via a certain multiplicative structure over $d$,
under additional conditions on $\nu$; this will be Theorem~\ref{THM:conditional-variance-evalulation-over-primes}.
The proofs of Theorems~\ref{THM:unconditional-variance-evalulation-over-integers} and~\ref{THM:conditional-variance-evalulation-over-primes} rest on delicate local calculations, and require particular care at primes $p\mid d$.
On a first reading, one may wish to focus on the simplest case $d=1$, which contains most of the key ideas.

\subsection{Non-archimedean work}

For integers $\enn,d\geq 1$, let
\begin{equation}
S^+_{\bm{0}}(\enn;d)
\defeq \sum_{\emm\geq 1:\, \lcm(\emm,d) = \enn}\,
\sum_{u\in (\ZZ/\emm\ZZ)^\times} \sum_{\bm{x}\in (\ZZ/\enn\ZZ)^6:\,
d\mid F_0(\bm{y}), F_0(\bm{z})} e_\emm(uF(\bm{x})).
\label{EQN:define-S^+_0(m;d)}
\end{equation}
Then $S^+_{\bm{0}}(\enn;d) = 0$ unless $d\mid \enn$.
For every positive integer $q\in d\ZZ$, we have by \eqref{EQN:define-mod-q-density-rho(q;d)}
\begin{equation*}
\begin{split}
\rho(q;d)
&= \sum_{a\in \ZZ/q\ZZ} q^{-6} \sum_{\bm{x}\in (\ZZ/q\ZZ)^6:\,
d\mid F_0(\bm{y}), F_0(\bm{z})} e_q(aF(\bm{x})) \\
&= \sum_{\emm\mid q} \sum_{u\in (\ZZ/\emm\ZZ)^\times} q^{-6} \sum_{\bm{x}\in (\ZZ/q\ZZ)^6:\,
d\mid F_0(\bm{y}), F_0(\bm{z})} e_\emm(uF(\bm{x})) \\
&= \sum_{\emm\geq 1:\, \lcm(\emm,d) \mid q}\,
\sum_{u\in (\ZZ/\emm\ZZ)^\times} q^{-6} \sum_{\bm{x}\in (\ZZ/q\ZZ)^6:\,
d\mid F_0(\bm{y}), F_0(\bm{z})} e_\emm(uF(\bm{x})),
\end{split}
\end{equation*}
since $\emm\mid q \Leftrightarrow \lcm(\emm,d)\mid q$ when $q\in d\ZZ$.
Reducing $\bm{x}$ modulo $\lcm(\emm,d)$, we get
(for $q\in d\ZZ$)
\begin{equation}
\begin{split}
\label{EQN:key-density-relation-to-S^+_0(m;d)}
\rho(q;d)
&= \sum_{\emm\geq 1:\, \lcm(\emm,d) \mid q}\,
\sum_{u\in (\ZZ/\emm\ZZ)^\times} \lcm(\emm,d)^{-6} \sum_{\bm{x}\in (\ZZ/\lcm(\emm,d)\ZZ)^6:\,
d\mid F_0(\bm{y}), F_0(\bm{z})} e_\emm(uF(\bm{x})) \\
&= \sum_{\enn\mid q} \enn^{-6}S^+_{\bm{0}}(\enn;d);
\end{split}
\end{equation}
this is well known when $d=1$
\cite{vaughan1997hardy}*{Lemma~2.12},
but perhaps less so for $d>1$.

\begin{proposition}
\label{PROP:standard-multiplicativity-properties-of-T_a,S^+_0}
Let $\enn_1,\enn_2,d_1,d_2\geq 1$ be integers.
If $\gcd(\enn_1,\enn_2) = 1$, then $T_a(\enn_1\enn_2) = T_a(\enn_1)T_a(\enn_2)$.
If $\gcd(\enn_1d_1, \enn_2d_2) = 1$, then $S^+_{\bm{0}}(\enn_1\enn_2;d_1d_2) = S^+_{\bm{0}}(\enn_1;d_1) S^+_{\bm{0}}(\enn_2;d_2)$.
\end{proposition}

\begin{proof}

It is well known that $T_a(\enn)$ and $S^+_{\bm{0}}(\enn;1)$ are multiplicative in $\enn$;
see \cite{browning2021cubic}*{Lemma~2.13} for a general treatment.
A similar proof should show that $S^+_{\bm{0}}(\enn;d)$ is multiplicative in pairs $(\enn,d)$, but we instead use double Dirichlet series.
By \eqref{EQN:key-density-relation-to-S^+_0(m;d)},
\begin{equation*}
\sum_{q,d\ge 1:\, d\mid q} \rho(q;d) q^{-r} d^{-s}
= \sum_{d,\enn\ge 1} \enn^{-6}S^+_{\bm{0}}(\enn;d) d^{-s} \sum_{q\ge 1:\, \enn\mid q} q^{-r}
= \zeta(r) \sum_{d,\enn\ge 1} \enn^{-6-r}S^+_{\bm{0}}(\enn;d) d^{-s}.
\end{equation*}
The left-hand side has an Euler product (by \eqref{EQN:define-mod-q-density-rho(q;d)} and the Chinese remainder theorem).
Multiplying by $\zeta(r)^{-1} = \prod_p (1-p^{-r})$ gives the desired multiplicativity of $S^+_{\bm{0}}(\enn;d)$.
\end{proof}

\begin{lemma}
\label{LEM:convenient-generic-bound-on-S^+_0(m;d)}
Let $\enn,d\geq 1$ be integers.
Suppose $v_p(\enn) > v_p(d)$ for all primes $p\mid \enn$.
Then $\abs{S^+_{\bm{0}}(\enn;d)}\leq \abs{S^+_{\bm{0}}(\enn;1)}$.
\end{lemma}

\begin{proof}
Here $\set{\emm\ge 1: \lcm(\emm,d) = \enn} = \set{\enn}$.
So by \eqref{EQN:define-S^+_0(m;d)}, the quantity $d^2 \cdot S^+_{\bm{0}}(\enn;d)$ equals
\begin{equation*}
\begin{split}
&d^2 \sum_{u\in (\ZZ/\enn\ZZ)^\times} \sum_{\bm{y}_1,\bm{y}_2\in (\ZZ/\enn\ZZ)^3:\, d\mid F_0(\bm{y}_1), F_0(\bm{y}_2)} e_\enn(uF_0(\bm{y}_1) - uF_0(\bm{y}_2)) \\
&= \sum_{u\in (\ZZ/\enn\ZZ)^\times} \sum_{\bm{y}_1,\bm{y}_2\in (\ZZ/\enn\ZZ)^3} e_\enn(uF_0(\bm{y}_1) - uF_0(\bm{y}_2))
\sum_{v_1,v_2\in \ZZ/d\ZZ} e_d(v_1F_0(\bm{y}_1) + v_2F_0(\bm{y}_2)) \\
&= \sum_{v_1,v_2\in \ZZ/d\ZZ} \sum_{u\in (\ZZ/\enn\ZZ)^\times}
\prod_{1\leq i\leq 2} \,\biggl(\,\sum_{\bm{y}_i\in (\ZZ/\enn\ZZ)^3} e_\enn((-1)^{i-1}uF_0(\bm{y}_i)) e_d(v_iF_0(\bm{y}_i))\biggr).
\end{split}
\end{equation*}
But for any $\eps\in \set{-1,1}$ and $v\in \ZZ/d\ZZ$, we have
\begin{equation*}
\sum_{u\in (\ZZ/\enn\ZZ)^\times} \,\Bigl\lvert{
\sum_{\bm{y}\in (\ZZ/\enn\ZZ)^3} e_\enn(\eps uF_0(\bm{y})) e_d(vF_0(\bm{y}))}
\Bigr\rvert^2
= S^+_{\bm{0}}(\enn;1),
\end{equation*}
since the formula $u\mapsto \eps u + (\enn/d) v$ defines a bijection on $(\ZZ/\enn\ZZ)^\times$.
So Cauchy over $u\in (\ZZ/\enn\ZZ)^\times$
implies $d^2\cdot \abs{S^+_{\bm{0}}(\enn;d)}
\le \sum_{v_1,v_2\in \ZZ/d\ZZ} S^+_{\bm{0}}(\enn;1)
= d^2\cdot S^+_{\bm{0}}(\enn;1)$.
This suffices.
\end{proof}

\begin{lemma}
\label{LEM:sum-S^+_0(m)-trivially}
Let $N,d\geq 1$ be integers.
Then the following holds:
\begin{equation*}
\sum_{\enn\in [N, 2N)} \enn^{-6}\, \abs{S^+_{\bm{0}}(\enn;d)}
\le d \sum_{ab \in [N, 2N):\, a\mid d} b^{-6} \, \abs{S^+_{\bm{0}}(b;1)}
\ll_\eps d^{5/3} N^{-2/3+\eps}.
\end{equation*}
\end{lemma}

\begin{proof}
Write $\enn=ef$, where $p\mid e\Rightarrow v_p(e)\le v_p(d)$ and $p\mid f\Rightarrow v_p(f)>v_p(d)$.
Trivially, $\abs{S^+_{\bm{0}}(e;\gcd(d,e^\infty))} \le \sum_{\emm\mid e} \phi(\emm) e^6 = e^7$.
By Lemma~\ref{LEM:convenient-generic-bound-on-S^+_0(m;d)}, $\abs{S^+_{\bm{0}}(f;\gcd(d,f^\infty))} \le \abs{S^+_{\bm{0}}(f;1)}$.
Therefore, $\abs{S^+_{\bm{0}}(\enn;d)} \leq e^7\, \abs{S^+_{\bm{0}}(f;1)}$ by Proposition~\ref{PROP:standard-multiplicativity-properties-of-T_a,S^+_0}.
Since $e\mid d$, we get
\begin{equation*}
\sum_{\enn\in [N, 2N)} \enn^{-6} \, \abs{S^+_{\bm{0}}(\enn;d)}
\le \sum_{e\mid d} e \sum_{f\in [N/e, 2N/e)} f^{-6} \, \abs{S^+_{\bm{0}}(f;1)}
\le d \sum_{e\mid d} \sum_{f\in [N/e, 2N/e)} f^{-6} \, \abs{S^+_{\bm{0}}(f;1)}.
\end{equation*}
By \cite{vaughan1997hardy}*{$(t,k,\lambda)=(4,3,0)$ case of Lemma~4.9} and \cite{vaughan1997hardy}*{Theorem~4.2},
the inner sum over $f$ is $\ll_\eps (N/e)^{\eps-2/3}$.
The lemma then follows from the bound $\sum_{e\mid d} e^{2/3-\eps} \ll d^{2/3}$.
\end{proof}

Let us note some standard consequences of our work so far.
First, the density $\sigma_p(d)$ from \eqref{EQN:define-local-densities-sigma_v} exists by \eqref{EQN:key-density-relation-to-S^+_0(m;d)} and Lemma~\ref{LEM:sum-S^+_0(m)-trivially}.
Second, the singular series $\mf{S}(d)$ converges, with
\begin{equation}
\label{EQN:F-singular-series-d}
\mf{S}(d) = \prod_{\textnormal{$p$ prime}} \sigma_p(d)
= \sum_{\enn\geq 1} \enn^{-6}S^+_{\bm{0}}(\enn;d)
= \sum_{\enn\geq 1:\, d\mid \enn} \enn^{-6}S^+_{\bm{0}}(\enn;d);
\end{equation}
here we use \eqref{EQN:key-density-relation-to-S^+_0(m;d)}, Proposition~\ref{PROP:standard-multiplicativity-properties-of-T_a,S^+_0}, and Lemma~\ref{LEM:sum-S^+_0(m)-trivially} to see that the infinite product and sums in \eqref{EQN:F-singular-series-d} all converge absolutely, and equal one another.


We now prove several results relating $T_a(-)$, from \eqref{EQN:define-T_a(m)}, to $S^+_{\bm{0}}(-;d)$, from \eqref{EQN:define-S^+_0(m;d)}.

\begin{lemma}
\label{LEM:double-T_b-key-calculations}
Let $d,m\geq 1$ be integers with $d\mid m$.
Then
\begin{align}
\sum_{n_1,n_2\geq 1:\, \lcm(n_1,d) = m,\; \lcm(n_2,d) = m}
\frac{1}{m} \sum_{b\in d\ZZ/m\ZZ} \frac{T_b(n_1)T_b(n_2)}{(n_1n_2)^3}
&= \frac{S^+_{\bm{0}}(m;d)}{m^6},
\label{EQN:double-T_b-lemma-first-goal} \\
\sum_{n_1,n_2\geq 1:\, \lcm(n_1,d) = m,\; \lcm(n_2,d) = m}
\,\biggl\lvert{\frac{1}{m} \sum_{b\in d\ZZ/m\ZZ} \frac{T_b(n_1)T_b(n_2)}{(n_1n_2)^3}}\biggr\rvert
&\leq \tau(m) \sum_{rn=m:\, r\mid d} \frac{\abs{S^+_{\bm{0}}(n;1)}}{n^6}.
\label{INEQ:double-T_b-lemma-second-goal}
\end{align}
\end{lemma}

\begin{proof}
By Proposition~\ref{PROP:standard-multiplicativity-properties-of-T_a,S^+_0}, it suffices to prove the lemma when $m$ is a prime power.
So suppose $(d,m) = (p^e,p^f)$, where $p$ is prime and $0\leq e\leq f$.

\emph{Case~1: $e<f$.}
Then $\set{n\geq 1: \lcm(n,d)=m} = \set{m}$.
In particular,
\begin{equation}
\begin{split}
\label{EQN:generic-simplify-S^+}
S^+_{\bm{0}}(m;d)
&= \sum_{u\in (\ZZ/m\ZZ)^\times}
\sum_{\bm{x}\in (\ZZ/m\ZZ)^6:\, d\mid F_0(\bm{y}), F_0(\bm{z})} e_m(uF(\bm{x})) \\
&= p^e \sum_{u\in (\ZZ/p^{f-e}\ZZ)^\times}
\sum_{\bm{x}\in (\ZZ/m\ZZ)^6:\, d\mid F_0(\bm{y}), F_0(\bm{z})} e_m(uF(\bm{x}))
\end{split}
\end{equation}
(since $e_m(ud)$ depends only on $u\bmod{p^{f-e}}$).
And if $b\in d\ZZ/m\ZZ$, then
\begin{equation}
\begin{split}
\label{EQN:introduce-d|F_0-implicit-in-T_b}
T_b(m) = \sum_{u\in (\ZZ/m\ZZ)^\times} \sum_{\bm{y}\in (\ZZ/m\ZZ)^3} e_m(uF_b(\bm{y}))
&= \sum_{u\in (\ZZ/m\ZZ)^\times} \sum_{\bm{y}\in (\ZZ/m\ZZ)^3:\, d\mid F_0(\bm{y})} e_m(uF_b(\bm{y})) \\
&= p^e \sum_{u\in (\ZZ/p^{f-e}\ZZ)^\times} \sum_{\bm{y}\in (\ZZ/m\ZZ)^3:\, d\mid F_0(\bm{y})} e_m(uF_b(\bm{y}))
\end{split}
\end{equation}
(since $\sum_{u\in (\ZZ/p^f\ZZ)^\times} e_{p^f}(uF_b(\bm{y})) = 0$ if $p^{f-1}\nmid F_b(\bm{y})$),
whence
\begin{equation*}
T_b(m)^2 = p^{2e} \sum_{u,v\in (\ZZ/p^{f-e}\ZZ)^\times} \sum_{\bm{y},\bm{z}\in (\ZZ/m\ZZ)^3:\, d\mid F_0(\bm{y}), F_0(\bm{z})} e_m(uF_b(\bm{y})+vF_b(\bm{z})).
\end{equation*}
But $\sum_{b\in d\ZZ/m\ZZ} e_m(-ub-vb) = p^{f-e}\cdot \bm{1}_{p^{f-e}\mid u+v}$.
So summing the previous display over $b\in d\ZZ/m\ZZ$, and then using \eqref{EQN:generic-simplify-S^+}, we get $\sum_{b\in d\ZZ/m\ZZ} T_b(m)^2 = p^f S^+_{\bm{0}}(m;d)$.
This suffices for both \eqref{EQN:double-T_b-lemma-first-goal}, \eqref{INEQ:double-T_b-lemma-second-goal}.
(For \eqref{INEQ:double-T_b-lemma-second-goal}, note that $\abs{S^+_{\bm{0}}(m;d)} \leq \abs{S^+_{\bm{0}}(m;1)}$ by Lemma~\ref{LEM:convenient-generic-bound-on-S^+_0(m;d)}).

\emph{Case~2: $e=f$.}
Then $d=m$ and $\set{n\geq 1: \lcm(n,d)=m} = \set{n\geq 1: n\mid m}$.
So
\begin{equation}
\label{EQN:special-simplify-S^+}
S^+_{\bm{0}}(m;d)
= \sum_{n\mid m} \sum_{u\in (\ZZ/n\ZZ)^\times} \sum_{\substack{\bm{x}\in (\ZZ/m\ZZ)^6: \\ m\mid F_0(\bm{y}), F_0(\bm{z})}} e_n(uF(\bm{x}))
= m \sum_{\substack{\bm{x}\in (\ZZ/m\ZZ)^6: \\ m\mid F_0(\bm{y}), F_0(\bm{z})}} 1.
\end{equation}
But $d\ZZ/m\ZZ = 0$ (in \eqref{EQN:double-T_b-lemma-first-goal}, \eqref{INEQ:double-T_b-lemma-second-goal}), and
\begin{equation}
\begin{split}
\label{EQN:special-simplify-T_0/n^3-divisor-sum}
\sum_{n\mid m} \frac{T_0(n)}{n^3}
&= \sum_{n\mid m} \sum_{u\in (\ZZ/n\ZZ)^\times} \EE_{\bm{y}\in (\ZZ/n\ZZ)^3}[e_n(uF_0(\bm{y}))] \\
&= \sum_{n\mid m} \sum_{u\in (\ZZ/n\ZZ)^\times} \EE_{\bm{y}\in (\ZZ/m\ZZ)^3}[e_n(uF_0(\bm{y}))] \\
&= \sum_{u\in \ZZ/m\ZZ} \EE_{\bm{y}\in (\ZZ/m\ZZ)^3}[e_m(uF_0(\bm{y}))]
= m^{-2} \sum_{\bm{y}\in (\ZZ/m\ZZ)^3:\, m\mid F_0(\bm{y})} 1.
\end{split}
\end{equation}
Upon squaring \eqref{EQN:special-simplify-T_0/n^3-divisor-sum}, dividing by $m$, and using \eqref{EQN:special-simplify-S^+}, we get \eqref{EQN:double-T_b-lemma-first-goal}.
For \eqref{INEQ:double-T_b-lemma-second-goal}, note that by Cauchy, $\abs{T_0(n)}^2 \leq \phi(n) \sum_{u\in (\ZZ/n\ZZ)^\times} \abs{\sum_{\bm{y}\in (\ZZ/n\ZZ)^3} e_n(uF_0(\bm{y}))}^2 = \phi(n) S^+_{\bm{0}}(n;1)$, so
\begin{equation}
\label{INEQ:special-T_0-L^2-bound}
\frac{1}{m} \sum_{n_1,n_2\mid m} \frac{\abs{T_0(n_1)T_0(n_2)}}{(n_1n_2)^3}
\leq \frac{\tau(m)}{m} \sum_{n\mid m} \frac{\abs{T_0(n)}^2}{n^6}
\leq \frac{\tau(m)}{m} \sum_{n\mid m} \frac{\phi(n) S^+_{\bm{0}}(n;1)}{n^6}.
\end{equation}
This suffices for \eqref{INEQ:double-T_b-lemma-second-goal} (since $\phi(n)\leq n\leq m$, and $(m/n)\mid m = d$).
\end{proof}

\begin{lemma}
\label{LEM:mixed-T_b-key-calculations}
Let $d,m\geq 1$ be integers with $d\mid m$.
Then
\begin{align}
\sum_{n\geq 1:\, \lcm(n,d) = m}
\frac{1}{m^3} \sum_{\bm{e}\in (\ZZ/m\ZZ)^3:\, F_0(\bm{e})\equiv 0\bmod{d}}
\frac{T_{F_0(\bm{e})}(n)}{n^3}
&= \frac{S^+_{\bm{0}}(m;d)}{m^6},
\label{EQN:mixed-T_b-first-goal} \\
\sum_{n\geq 1:\, \lcm(n,d) = m}\,
\biggl\lvert{\frac{1}{m^3} \sum_{\bm{e}\in (\ZZ/m\ZZ)^3:\, F_0(\bm{e})\equiv 0\bmod{d}}
\frac{T_{F_0(\bm{e})}(n)}{n^3}}\biggr\rvert
&\leq d\cdot \tau(m) \sum_{rn=m:\, r\mid d} \frac{\abs{S^+_{\bm{0}}(n;1)}}{n^6}.
\label{INEQ:mixed-T_b-second-goal}
\end{align}
\end{lemma}

\begin{proof}
Again,
we may assume that $(d,m) = (p^e,p^f)$, where $p$ is prime and $0\leq e\leq f$.

\emph{Case~1: $e<f$.}
Then $\set{n\geq 1: \lcm(n,d)=m} = \set{m}$.
But by \eqref{EQN:introduce-d|F_0-implicit-in-T_b},
\begin{equation*}
\sum_{\bm{e}\in (\ZZ/m\ZZ)^3:\, F_0(\bm{e})\equiv 0\bmod{d}} T_{F_0(\bm{e})}(m)
= \sum_{u\in (\ZZ/m\ZZ)^\times} \sum_{\bm{e},\bm{y}\in (\ZZ/m\ZZ)^3:\,
d\mid F_0(\bm{e}), F_0(\bm{y})} e_m(u(F_0(\bm{y})-F_0(\bm{e}))),
\end{equation*}
which equals $S^+_{\bm{0}}(m;d)$ by \eqref{EQN:generic-simplify-S^+}.
Both \eqref{EQN:mixed-T_b-first-goal}, \eqref{INEQ:mixed-T_b-second-goal} follow (in the present case).

\emph{Case~2: $e=f$.}
Then $d=m$ and $\set{n\geq 1: \lcm(n,d)=m} = \set{n\geq 1: n\mid m}$.
So \eqref{EQN:mixed-T_b-first-goal} follows from \eqref{EQN:special-simplify-T_0/n^3-divisor-sum} and \eqref{EQN:special-simplify-S^+}.
Also, \eqref{INEQ:mixed-T_b-second-goal} follows from \eqref{INEQ:special-T_0-L^2-bound} (since $T_0(n) = T_0(n) T_0(1)$ for all $n\ge 1$, and we have $\#\set{\bm{e}\in (\ZZ/m\ZZ)^3: F_0(\bm{e})\equiv 0\bmod{d}} \leq m^3 = m^2 d$ trivially).
\end{proof}

\begin{lemma}
\label{LEM:unbalanced-T_b-moment-vanishing}
Suppose $n_1,n_2,d\geq 1$ are integers with $\lcm(n_1,d)\neq \lcm(n_2,d)$.
Then
\begin{equation}
\label{EQN:unbalanced-T_b-moment-vanishing}
\sum_{b\in d\ZZ/n_1n_2d\ZZ} T_b(n_1)T_b(n_2) = 0.
\end{equation}
\end{lemma}

\begin{proof}
We may assume $\lcm(n_1,d) < \lcm(n_2,d)$.
Let $r\defeq \lcm(n_1,d)$; then $n_2\nmid r$.
So $\sum_{a\in \gcd(n_2,r)\ZZ/n_2\ZZ} e_{n_2}(-ua) = 0$ for all $u\in (\ZZ/n_2\ZZ)^\times$.
Thus $\sum_{b\in \ZZ/n_1n_2d\ZZ:\, b\equiv c\bmod{r}} T_b(n_2) = 0$ for all $c\in \ZZ/r\ZZ$.
Multiplying by $T_c(n_1)$ and summing over $c\in d\ZZ/r\ZZ$, we get \eqref{EQN:unbalanced-T_b-moment-vanishing}.
\end{proof}

\begin{proposition}
\label{PROP:truncated-non-archimedean-local-probabilistic-calculation}
Let $d,K\geq 1$ be integers.
Then
\begin{align}
\sum_{n_1,n_2\leq K} \frac{1}{n_1n_2d} \sum_{b\in d\ZZ/n_1n_2d\ZZ} \frac{T_b(n_1)T_b(n_2)}{(n_1n_2)^3}
&= \mf{S}(d) + O_\eps(d^{5/3} K^{-2/3+\eps}),
\label{EQN:pure-T_b-L^2-moment} \\
\sum_{n\leq K} \frac{1}{(nd)^3} \sum_{\bm{e}\in (\ZZ/nd\ZZ)^3:\, F_0(\bm{e})\equiv 0\bmod{d}} \frac{T_{F_0(\bm{e})}(n)}{n^3}
&= \mf{S}(d) + O_\eps(d^{5/3} K^{-2/3+\eps}).
\label{EQN:mixed-T_b-L^1-moment}
\end{align}
\end{proposition}

\begin{proof}

If $n_1,n_2\geq 1$ are integers with $\lcm(n_1,d) = \lcm(n_2,d) = m$, say, then $d\mid m$ and
\begin{equation*}
\frac{1}{n_1n_2d}
\sum_{b\in d\ZZ/n_1n_2d\ZZ} \frac{T_b(n_1)T_b(n_2)}{(n_1n_2)^3}
= \frac{1}{m} \sum_{b\in d\ZZ/m\ZZ} \frac{T_b(n_1)T_b(n_2)}{(n_1n_2)^3}.
\end{equation*}
By Lemmas~\ref{LEM:unbalanced-T_b-moment-vanishing} and~\ref{LEM:double-T_b-key-calculations}, it follows that the left-hand side of \eqref{EQN:pure-T_b-L^2-moment} equals
\begin{equation*}
\begin{split}
&\sum_{m\geq 1:\, d\mid m}
\sum_{\substack{n_1,n_2\leq K: \\ \lcm(n_1,d) = m,\; \lcm(n_2,d) = m}}
\frac{1}{m} \sum_{b\in d\ZZ/m\ZZ} \frac{T_b(n_1)T_b(n_2)}{(n_1n_2)^3} \\
&= \sum_{m\leq K:\, d\mid m} \frac{S^+_{\bm{0}}(m;d)}{m^6}
+ \sum_{m>K:\, d\mid m} \tau(m) \sum_{rn=m:\, r\mid d} \frac{O(\abs{S^+_{\bm{0}}(n;1)})}{n^6} \\
&= \mf{S}(d) + O_\eps(d^{5/3} K^{-2/3+\eps}),
\end{split}
\end{equation*}
where in the final step we write $\sum_{m\le K:\, d\mid m} = \sum_{m\ge 1:\, d\mid m} - \sum_{m>K:\, d\mid m}$, use \eqref{EQN:F-singular-series-d} to write $\sum_{m\ge 1:\, d\mid m} S^+_{\bm{0}}(m;d)/m^6 = \mf{S}(d)$, and use Lemma~\ref{LEM:sum-S^+_0(m)-trivially} to bound the $m>K$ contributions.

The second part, \eqref{EQN:mixed-T_b-L^1-moment}, is similar, but simpler.
If $n\geq 1$ and $\lcm(n,d) = m$, say, then
\begin{equation*}
\frac{1}{(nd)^3} \sum_{\bm{e}\in (\ZZ/nd\ZZ)^3:\, F_0(\bm{e})\equiv 0\bmod{d}} \frac{T_{F_0(\bm{e})}(n)}{n^3}
= \frac{1}{m^3} \sum_{\bm{e}\in (\ZZ/m\ZZ)^3:\, F_0(\bm{e})\equiv 0\bmod{d}} \frac{T_{F_0(\bm{e})}(n)}{n^3}.
\end{equation*}
By Lemma~\ref{LEM:mixed-T_b-key-calculations}, the left-hand side of \eqref{EQN:mixed-T_b-L^1-moment} thus equals
\begin{equation*}
\begin{split}
&\sum_{m\geq 1:\, d\mid m}
\sum_{\substack{n\leq K: \\ \lcm(n,d) = m}}
\frac{1}{m^3} \sum_{\bm{e}\in (\ZZ/m\ZZ)^3:\, F_0(\bm{e})\equiv 0\bmod{d}} \frac{T_{F_0(\bm{e})}(n)}{n^3} \\
&= \sum_{m\leq K:\, d\mid m} \frac{S^+_{\bm{0}}(m;d)}{m^6}
+ \sum_{m>K:\, d\mid m} d\cdot \tau(m) \sum_{rn=m:\, r\mid d} \frac{O(\abs{S^+_{\bm{0}}(n;1)})}{n^6} \\
&= \mf{S}(d) + O_\eps(d^{5/3} K^{-2/3+\eps}),
\end{split}
\end{equation*}
where in the last step we again use \eqref{EQN:F-singular-series-d} and Lemma~\ref{LEM:sum-S^+_0(m)-trivially}.
This completes the proof.
\end{proof}

\subsection{Archimedean work}

Let $\nu\in C^\infty_c(\RR^3)$, and suppose $\bm{0}\notin \Supp\nu$.
Given $X\in \RR_{>0}$ and $(\bm{y},a)\in \RR^3\times\RR$,
write $\tilde{\bm{y}}\defeq \bm{y}/X$ and $\tilde{a}\defeq a/X^3$.
By \eqref{EQN:define-F_a-real-density} (after replacing $\eps$ with $X^3\eps$),
\begin{equation}
\label{EQN:rescale-real-density-integral}
\sigma_{\infty,a,\nu}(X)
= \lim_{\eps\to 0}{(2\eps)^{-1}
\int_{\tilde{\bm{y}}\in \RR^3:\, \abs{F_0(\tilde{\bm{y}})-\tilde{a}}\leq\eps}
d\tilde{\bm{y}}\, \nu(\tilde{\bm{y}})
= \sigma_{\infty,\tilde{a},\nu}(1)}.
\end{equation}

For convenience
(when working with $\sigma_{\infty,a,\nu}(X)$),
we now observe the following:
\begin{observation}
Suppose $\abs{y_1}\geq \delta>0$ for all $\bm{y}\in \Supp\nu$.
Let $(a,X)\in \RR\times\RR_{>0}$.
Then a change of variables in \eqref{EQN:rescale-real-density-integral} from $\tilde{y}_1$ to
$F_0\defeq F_0(\tilde{\bm{y}})$ gives (by \cite{chambert2010igusa}*{\S5.4, par.~4})
\begin{equation}
\label{EQN:surface-integral-representation-of-real-density}
\sigma_{\infty,a,\nu}(X)
= \int_{\RR^2} d\tilde{y}_2\,d\tilde{y}_3\,\nu(\tilde{\bm{y}})
\cdot \abs{\partial F_0/\partial\tilde{y}_1}^{-1},
\end{equation}
where $F_0,\tilde{y}_1$ are constrained by the equation $F_0=\tilde{a}$ (and thus determined by $\tilde{y}_1, \tilde{y}_2$),
and where $\partial F_0/\partial\tilde{y}_1 = 3\tilde{y}_1^2\geq 3\delta^2>0$ over the support of the integrand.
\end{observation}

At least in the absence of better surface coordinates,
the definition \eqref{EQN:define-F_a-real-density} (via ``$\eps$-thickening'') still provides greater intuition,
while the surface integral allows for effortless rigor.
The area form $d\tilde{y}_2\,d\tilde{y}_3\,\abs{\partial F_0/\partial\tilde{y}_1}^{-1}$ in \eqref{EQN:surface-integral-representation-of-real-density} is often called a \emph{Leray form}, as in \cite{peyre1995hauteurs}.

We now prove three results on real densities.
For technical convenience, we assume for the rest of \S\ref{SEC:general-variance-setup-and-estimation} that $\nu$ is clean (see \eqref{EQN:condition-for-clean}).
Let $B(\nu)$ and $\norm{\nu}_{k,\infty}$ be as in \S\ref{SUBSEC:conventions}.

\begin{proposition}
\label{PROP:uniform-bound-on-derivatives-of-densities}
For integers $k\geq 0$,
we have $\partial_{\tilde{a}}^k[\sigma_{\infty,a,\nu}(X)]\ll_k \norm{\nu}_{k,\infty} B(\nu)^{A_0+A_1k}$ (for some constants $A_0, A_1\in [1, 10]$), uniformly over $(a, X)\in \RR\times \RR_{>0}$.
\end{proposition}

\begin{proof}
The integrand on the right in \eqref{EQN:surface-integral-representation-of-real-density} vanishes unless $\tilde{a}\ll B(\nu)^3$ and $\tilde{y}_1^{-1},\tilde{y}_2,\tilde{y}_3\ll B(\nu)$.
Fix $\tilde{y}_2,\tilde{y}_3$,
and let $\tilde{y}_1$ vary with $\tilde{a}$ according to $F_0=\tilde{a}$.
Then $\partial_{\tilde{a}}[\tilde{y}_1] = (3\tilde{y}_1^2)^{-1}\ll B(\nu)^2$.
Now repeatedly apply $\partial_{\tilde{a}}$ to the integrand
(using Leibniz and the chain rule).
This gives $\partial_{\tilde{a}}^k[\sigma_{\infty,a,\nu}(X)]
\ll_k \int_{\tilde{y}_2,\tilde{y}_3\ll B(\nu)} d\tilde{y}_2\,d\tilde{y}_3\,
\norm{\nu}_{k,\infty} B(\nu)^{2k} / \abs{\tilde{y}_1}^{2+k}
\ll_k \norm{\nu}_{k,\infty} B(\nu)^{4+3k}$.
\end{proof}

\begin{proposition}
\label{PROP:real-local-probabilistic-calculation}
Let $X\in \RR_{>0}$.
The ``pure $L^2$ moment''
$\int_{a\in\RR}d\tilde{a}\,\sigma_{\infty,a,\nu}(X)^2$
and the ``mixed $L^1$ moment''
$\int_{\tilde{\bm{z}}\in\RR^3}d\tilde{\bm{z}}\,\nu(\tilde{\bm{z}})\sigma_{\infty,F_0(\bm{z}),\nu}(X)$
both equal $\sigma_{\infty,\nu^{\otimes 2}}$ (defined via \eqref{EQN:define-local-densities-sigma_v}).
\end{proposition}

\begin{proof}
First,
$\int_{a\in\RR}d\tilde{a}\,\sigma_{\infty,a,\nu}(X)^2$ expands (via \eqref{EQN:surface-integral-representation-of-real-density} and the relation $\tilde{a}=F_0(\tilde{\bm{y}})$) to
\begin{equation*}
\int_{\RR^4} d\tilde{y}_2\,d\tilde{y}_3\,d\tilde{z}_2\,d\tilde{z}_3
\int_{\tilde{y}_1\in\RR}dF_0(\tilde{\bm{y}})\,
\frac{\nu(\tilde{\bm{y}})\nu(\tilde{\bm{z}})}{\partial_{\tilde{y}_1}F_0\vert_{\tilde{\bm{y}}}}
(\partial_{\tilde{z}_1}F_0\vert_{\tilde{\bm{z}}})^{-1}\vert_{F_0(\tilde{\bm{y}})=F_0(\tilde{\bm{z}})},
\end{equation*}
which simplifies to
$\int_{\RR^5} d\tilde{y}_2\,d\tilde{y}_3\,d\tilde{z}_2\,d\tilde{z}_3\,d\tilde{y}_1\,
\nu^{\otimes 2}(\tilde{\bm{x}})\cdot(\partial_{\tilde{z}_1}F_0\vert_{\tilde{\bm{z}}})^{-1}\vert_{F(\tilde{\bm{x}})=0}$ (by \eqref{EQN:define-symmetric-separable-w-weight-from-nu}),
which equals $\sigma_{\infty,\nu^{\otimes 2}}$
by \cite{chambert2010igusa}*{\S5.4, par.~4}.
Second,
$F_0(\bm{z})/X^3 = F_0(\tilde{\bm{z}})$,
so by \eqref{EQN:surface-integral-representation-of-real-density},
\begin{equation*}
\int_{\tilde{\bm{z}}\in\RR^3}d\tilde{\bm{z}}\,\nu(\tilde{\bm{z}})\sigma_{\infty,F_0(\bm{z}),\nu}(X)
= \int_{\RR^3\times\RR^2}d\tilde{\bm{z}}\,d\tilde{y}_2\,d\tilde{y}_3\,\nu(\tilde{\bm{y}})\nu(\tilde{\bm{z}})
\cdot(\partial_{\tilde{y}_1}F_0\vert_{\tilde{\bm{y}}})^{-1}\vert_{F_0(\tilde{\bm{y}})=F_0(\tilde{\bm{z}})},
\end{equation*}
which again simplifies to $\sigma_{\infty,\nu^{\otimes 2}}$.
\end{proof}


\begin{proposition}
\label{PROP:replace-dense-sum-with-integral}
Let $X,N\geq 1$ be integers.
Let $b\in \ZZ/N\ZZ$ and $\bm{e}\in (\ZZ/N\ZZ)^3$.
Then for integers $j\geq 4$, we have (for some constants $A_2,\dots,A_5\in [1, 30]$)
\begin{align}
\sum_{a\in \ZZ:\, a\equiv b\bmod{N}} \sigma_{\infty,a,\nu}(X)^2
&= \frac{X^3 \sigma_{\infty,\nu^{\otimes 2}}}{N}
+ \frac{O_j(\norm{\nu}_{j,\infty}^2 B(\nu)^{A_2+A_3j})}{(X^3/N)^{j-1}},
\label{EQN:first-Poisson-sum-goal} \\
\sum_{\bm{z}\in \ZZ^3:\, \bm{z}\equiv \bm{e}\bmod{N}} \nu(\bm{z}/X)
\sigma_{\infty,F_0(\bm{z}),\nu}(X)
&= \frac{X^3 \sigma_{\infty,\nu^{\otimes 2}}}{N^3}
+ \frac{O_j(\norm{\nu}_{j,\infty}^2 B(\nu)^{A_4+A_5j})}{(X/N)^{j-3}},
\label{EQN:second-Poisson-sum-goal} \\
\sum_{\bm{z}\in \ZZ^3:\, \bm{z}\equiv \bm{e}\bmod{N}} \nu(\bm{z}/X)^2
&= \frac{X^3 \norm{\nu}_{L^2(\RR^3)}^2}{N^3}
+ \frac{O_j(\norm{\nu}_{j,\infty}^2 B(\nu)^3)}{(X/N)^{j-3}}.
\label{EQN:third-Poisson-sum-goal}
\end{align}
\end{proposition}

\begin{proof}
We want to replace certain sums with integrals.
This is loosely analogous to \cite{diaconu2019admissible}*{proof of Lemma~3.1},
but we can get better error terms using the smoothness of $\nu$.
The error will depend polynomially on the size and support of $\nu$.
We do not optimize the exponents of $B(\nu)$ above, since they are ultimately not so important.

The method is standard.
Poisson summation,
together with Proposition~\ref{PROP:real-local-probabilistic-calculation},
gives
\begin{equation*}
\sum_{a\equiv b\bmod{N}} \sigma_{\infty,a,\nu}(X)^2
= \frac{X^3 \sigma_{\infty,\nu^{\otimes 2}}}{N}
+ \sum_{c\neq 0} \frac{O(1)}{N} \left\lvert{
\int_{a\in \RR} da\,\sigma_{\infty,a,\nu}(X)^2 e(-c\cdot a/N)
}\right\rvert.
\end{equation*}
We plug in $a = X^3\tilde{a}$,
integrate by parts $j\ge 2$ times in $\tilde{a}$,
and invoke Proposition~\ref{PROP:uniform-bound-on-derivatives-of-densities},
to get
\begin{equation*}
\frac{1}{N} \left\lvert{
\int_{a\in \RR} da\,\sigma_{\infty,a,\nu}(X)^2 e(-c\cdot a/N)
}\right\rvert
\ll_j \frac{\norm{\nu}_{j,\infty}^2 B(\nu)^{A_2+A_3j}}{\abs{c}^j (X^3/N)^{j-1}}
\quad (A_2 = 3+2A_0,\; A_3 = A_1).
\end{equation*}
(Note that $\sigma_{\infty,a,\nu}(X)=0$ unless $\tilde{a}\ll B(\nu)^3$.)
The estimate \eqref{EQN:first-Poisson-sum-goal} follows.

The estimate \eqref{EQN:second-Poisson-sum-goal} similar.
Let $\bm{c}\in \ZZ^3\setminus \set{\bm{0}}$.
Choose $i\in \set{1,2,3}$ with $\abs{c_i} = \norm{\bm{c}}$.
Integrating by parts $j\ge 4$ times in $\tilde{z}_i = z_i/X$ gives
(via Proposition~\ref{PROP:uniform-bound-on-derivatives-of-densities} with $a=F_0(\bm{z})$)
\begin{equation*}
\frac{1}{N^3} \left\lvert{
\int_{\bm{z}\in \RR^3} d\bm{z}\,\nu(\bm{z}/X)
\sigma_{\infty,F_0(\bm{z}),\nu}(X) e(-\bm{c}\cdot \bm{z}/N)
}\right\rvert
\ll_j \frac{\norm{\nu}_{j,\infty}^2 B(\nu)^{A_4+A_5j}}{\norm{\bm{c}}^j (X/N)^{j-3}},
\end{equation*}
with $A_4 = 3+A_0$, $A_5 = A_1+2$; this is because a factor of $\partial_{\tilde{z}_i} F_0(\tilde{\bm{z}}) = 3\tilde{z}_i^2 \ll B(\nu)^2$ appears each time we differentiate $\sigma_{\infty,F_0(\bm{z}),\nu}(X)$, by the chain rule.
(Note that $\nu(\bm{z}/X)=0$ unless $\tilde{\bm{z}}\ll B(\nu)$.)
Poisson summation and Proposition~\ref{PROP:real-local-probabilistic-calculation} then give \eqref{EQN:second-Poisson-sum-goal}.

Finally, \eqref{EQN:third-Poisson-sum-goal} is similar to \eqref{EQN:second-Poisson-sum-goal}, but easier; integrating by parts in $\tilde{z}_i = z_i/X$ gives
\begin{equation*}
\frac{1}{N^3} \left\lvert{
\int_{\bm{z}\in \RR^3} d\bm{z}\,\nu(\bm{z}/X)^2 e(-\bm{c}\cdot \bm{z}/N)
}\right\rvert
\ll_j \frac{\norm{\nu}_{j,\infty}^2 B(\nu)^3}{\norm{\bm{c}}^j (X/N)^{j-3}},
\end{equation*}
for $j\ge 4$.
Poisson summation and \eqref{EQN:define-L2-norm-of-nu} then give \eqref{EQN:third-Poisson-sum-goal}.
\end{proof}

\subsection{Final calculations}

We are finally ready to establish the main results of \S\ref{SEC:general-variance-setup-and-estimation}, which concerns the variance defined in \eqref{EQN:define-approximate-variance}.

\begin{theorem}
\label{THM:unconditional-variance-evalulation-over-integers}
Let $\nu\in C^\infty_c(\RR^3)$ satisfy \eqref{EQN:condition-for-clean}.
Let $X,K,d\geq 1$ be integers with $Kd\leq X^{9/10}$.
Then
\begin{equation*}
\map{Var}(X,K;d) =
[N_{\nu^{\otimes 2}}(X;d) - \mf{S}(d) \cdot \sigma_{\infty,\nu^{\otimes 2}} X^3]
+ O_\eps(d^{5/3} K^{-2/3+\eps} X^3 \norm{\nu}_{100,\infty}^2 B(\nu)^{5000}).
\end{equation*}
\end{theorem}

\begin{proof}
Squaring out \eqref{EQN:define-approximate-variance} yields $\map{Var}(X,K;d) = \Sigma_1 - 2\Sigma_2 + \Sigma_3$, where
\begin{equation*}
\Sigma_1\defeq \sum_{a\in d\ZZ} N_{a,\nu}(X)^2, \quad
\Sigma_2\defeq \sum_{a\in d\ZZ} N_{a,\nu}(X)s_a(K)\sigma_{\infty,a,\nu}(X), \quad
\Sigma_3\defeq \sum_{a\in d\ZZ} [s_a(K)\sigma_{\infty,a,\nu}(X)]^2.
\end{equation*}
Plugging in \eqref{EQN:define-integral-point-count-N_a,nu(X)} gives $\Sigma_1 = N_{\nu^{\otimes 2}}(X;d)$ (by \eqref{EQN:define-fiber-product-form-F-from-F_0}, \eqref{EQN:define-integral-point-count-N_w(X;d)}, and \eqref{EQN:define-symmetric-separable-w-weight-from-nu}).
Next, write
\begin{equation*}
\Sigma_3
= \sum_{n_1,n_2\leq K} \sum_{b\in d\ZZ/n_1n_2d\ZZ} (n_1n_2)^{-3}T_b(n_1)T_b(n_2)
\sum_{a\in \ZZ:\, a\equiv b\bmod{n_1n_2d}} \sigma_{\infty,a,\nu}(X)^2
\end{equation*}
by plugging in \eqref{EQN:define-s_a(K)} and then grouping terms by $a\bmod{n_1n_2d}$;
and write
\begin{equation*}
\begin{split}
\Sigma_2
&= \sum_{\bm{z}\in \ZZ^3:\, F_0(\bm{z})\equiv 0\bmod{d}} \nu(\bm{z}/X)s_{F_0(\bm{z})}(K)\sigma_{\infty,F_0(\bm{z}),\nu}(X) \\
&= \sum_{n\leq K} \sum_{\bm{e}\in (\ZZ/nd\ZZ)^3:\, F_0(\bm{e})\equiv 0\bmod{d}} n^{-3}T_{F_0(\bm{e})}(n)
\sum_{\bm{z}\in \ZZ^3:\, \bm{z}\equiv \bm{e}\bmod{nd}} \nu(\bm{z}/X)\sigma_{\infty,F_0(\bm{z}),\nu}(X)
\end{split}
\end{equation*}
by expanding $N_{a,\nu}(X)$, plugging in \eqref{EQN:define-s_a(K)}, and grouping terms by $\bm{z}\bmod{nd}$.
Then by Proposition~\ref{PROP:replace-dense-sum-with-integral} and the trivial bound $\abs{T_b(n)}\leq n^4$, we have (for $j\geq 4$)
\begin{equation*}
\Sigma_3
- \sum_{n_1,n_2\leq K} \sum_{b\in d\ZZ/n_1n_2d\ZZ} (n_1n_2)^{-3}T_b(n_1)T_b(n_2)
\cdot \frac{X^3\sigma_{\infty,\nu^{\otimes 2}}}{n_1n_2d}
\ll K^6 \cdot \frac{O_j(\norm{\nu}_{j,\infty}^2 B(\nu)^{A_2+A_3j})}{(X^3/K^2d)^{j-1}}
\end{equation*}
by \eqref{EQN:first-Poisson-sum-goal} (with $N=n_1n_2d$), and
\begin{equation*}
\Sigma_2
- \sum_{n\leq K} \sum_{\bm{e}\in (\ZZ/nd\ZZ)^3:\, F_0(\bm{e})\equiv 0\bmod{d}} n^{-3}T_{F_0(\bm{e})}(n)
\cdot \frac{X^3\sigma_{\infty,\nu^{\otimes 2}}}{(nd)^3}
\ll K^2(Kd)^3 \cdot \frac{O_j(\norm{\nu}_{j,\infty}^2 B(\nu)^{A_4+A_5j})}{(X/Kd)^{j-3}}
\end{equation*}
by \eqref{EQN:second-Poisson-sum-goal} (with $N=nd$), where $A_2,\dots,A_5\le 30$.
Upon taking $j=100$ above, and plugging in \eqref{EQN:pure-T_b-L^2-moment} and \eqref{EQN:mixed-T_b-L^1-moment} from Proposition~\ref{PROP:truncated-non-archimedean-local-probabilistic-calculation}, we find that $\map{Var}(X,K;d)$ equals
\begin{equation*}
[N_{\nu^{\otimes 2}}(X;d) - \mf{S}(d) \cdot \sigma_{\infty,\nu^{\otimes 2}} X^3]
+ O_\eps(d^{5/3} K^{-2/3+\eps}\cdot \sigma_{\infty,\nu^{\otimes 2}} X^3)
+ O(\norm{\nu}_{100,\infty}^2 B(\nu)^{5000}),
\end{equation*}
since $X,K,d,B(\nu)\geq 1$ and $Kd\leq X^{9/10}$ (and thus $K^6$, $K^2d$, $K^2(Kd)^3$ are at most $X^6$, $X^2$, $X^5$, respectively).
To finish, note that $\sigma_{\infty,\nu^{\otimes 2}} \ll B(\nu)^3\cdot \norm{\nu}_{0,\infty}^2 B(\nu)^{2A_0}$ by Propositions~\ref{PROP:real-local-probabilistic-calculation} and~\ref{PROP:uniform-bound-on-derivatives-of-densities},
and that $\norm{\nu}_{100,\infty}^2 B(\nu)^{5000} \le d^{5/3} K^{-2/3+\eps} X^3 \norm{\nu}_{100,\infty}^2 B(\nu)^{5000}$ trivially.
\end{proof}

Call a function $\nu\maps \RR^3\to \RR$ \emph{very clean} if
\begin{equation}
\label{EQN:condition-for-very-clean}
(\Supp{\nu}) \cap \set{\bm{y}\in \RR^3: y_1y_2y_3(y_1+y_2)(y_1+y_3)(y_2+y_3) = 0} = \emptyset,
\end{equation}
and \emph{$\map{LinAut}(F_0)$-symmetric} if
\begin{equation}
\label{EQN:condition-for-symmetric}
\nu(y_1,y_2,y_3) = \nu(y_{\sigma(1)},y_{\sigma(2)},y_{\sigma(3)})
\quad\textnormal{for all }\sigma\in S_3.
\end{equation}

\begin{theorem}
\label{THM:conditional-variance-evalulation-over-primes}
Let $\xi\in \set{0,1}$.
Fix $\delta\in (0, 9/10)$ and an integer $k\geq 5000$, and assume \eqref{EQN:hard-HLH-clean-weight-level-d} for $\xi$.
Let $\nu\in C^\infty_c(\RR^3)$ satisfy \eqref{EQN:condition-for-very-clean} and \eqref{EQN:condition-for-symmetric}.
Let $X,K\geq 2$ be integers with $K\leq X^{9/10-\delta}$.
Fix $\hbar \in (0, 9\delta/20]$,
and let $P$ be the product of all primes $p<X^{\xi\hbar}$.
Then
\begin{equation*}
\sum_{a\in \ZZ:\, \gcd(a,P) = 1}[N_{a,\nu}(X) - s_a(K)\sigma_{\infty,a,\nu}(X)]^2
\ll \frac{X^3 \norm{\nu}_{L^2(\RR^3)}^2}{(\log{X})^\xi} + \frac{X^3 \norm{\nu}_{k,\infty}^2 B(\nu)^k}{\min(X^{\delta/11}, K^{2/3}X^{-3\delta})}.
\end{equation*}
\end{theorem}


\begin{proof}
We apply \eqref{EQN:hard-HLH-clean-weight-level-d} with $w = \nu^{\otimes 2}$.
Note that $\norm{\nu^{\otimes 2}}_{k,\infty} \leq \norm{\nu}_{k,\infty}^2$ by \eqref{EQN:define-Sobolev-norm},
and $B(\nu^{\otimes 2})\leq B(\nu)$ since $\Supp(\nu^{\otimes 2}) \belongs (\Supp \nu)^2$.
Now let $d\leq X^{\xi\delta}$ (so that $Kd\le X^{9/10}$).
By Theorem~\ref{THM:unconditional-variance-evalulation-over-integers}, \eqref{EQN:hard-HLH-clean-weight-level-d}, and \eqref{EQN:define-HLH-error-E_w(X;d)}, we conclude that (since $k\ge 5000$, and $d^{5/3} K^\eps \le X^{2\delta}$ for $\eps = \delta/3$, say)
\begin{equation*}
\map{Var}(X,K;d)
= \sum_{\textnormal{special }\bm{x}\in \ZZ^6:\, d\mid F_0(\bm{y}),F_0(\bm{z})} \nu^{\otimes 2}(\bm{x}/X)
+ O(X^{3-\delta} + X^{3+2\delta} K^{-2/3})\cdot \norm{\nu}_{k,\infty}^2 B(\nu)^k.
\end{equation*}
By \eqref{EQN:condition-for-very-clean} and \eqref{COND:x-is-in-the-union-of-special-linear-subspaces}, any special $\bm{x}\in \ZZ^6$ with $\nu^{\otimes 2}(\bm{x}/X)\ne 0$ must satisfy $(y_1,y_2,y_3) = (z_{\sigma(1)},z_{\sigma(2)},z_{\sigma(3)})$ for some $\sigma\in S_3$,
since $\prod_{1\le i<j\le 3} (y_i+y_j)(z_i+z_j) \ne 0$.
So, by \eqref{EQN:condition-for-symmetric},
\begin{equation}
\label{EQN:linear-space-inclusion-exclusion}
\sum_{\textnormal{special }\bm{x}\in \ZZ^6:\, d\mid F_0(\bm{y}),F_0(\bm{z})} \nu^{\otimes 2}(\bm{x}/X)
= 3! \sum_{\bm{y}\in \ZZ^3:\, d\mid F_0(\bm{y})} \nu(\bm{y}/X)^2
+ O(\norm{\nu}_{0,\infty}^2\cdot [B(\nu) X]^2).
\end{equation}
But $\nu\in C^\infty_c(\RR^3)$, so by \eqref{EQN:third-Poisson-sum-goal} (with $N=d$),
\begin{equation}
\label{EQN:special-solution-Poisson-summation}
\sum_{\bm{y}\in \ZZ^3:\, d\mid F_0(\bm{y})} \nu(\bm{y}/X)^2
= \sum_{\bm{e}\in (\ZZ/d\ZZ)^3:\, F_0(\bm{e})\equiv 0\bmod{d}}
\left( \frac{X^3 \norm{\nu}_{L^2(\RR^3)}^2}{d^3}
+ \frac{O_k(\norm{\nu}_{k,\infty}^2 B(\nu)^3)}{(X/d)^{k-3}} \right),
\end{equation}
since $k\geq 4$.
For integers $n\geq 1$, let
\begin{equation*}
    g(n)\defeq n^{-3} \cdot \card{\set{\bm{e}\in (\ZZ/n\ZZ)^3: F_0(\bm{e})\equiv 0\bmod{n}}}.
\end{equation*}
Then our work above (in the the last several displays) implies
\begin{equation}
\label{EQN:final-conditional-estimate-for-variance-restricted-to-dZ}
\map{Var}(X,K;d) = 3! g(d) X^3 \norm{\nu}_{L^2(\RR^3)}^2 + r_d,
\end{equation}
where $r_d \ll (X^{3-\delta} + X^{3+2\delta} K^{-2/3} + X^2 + d^3/(X/d)^{k-3})\cdot \norm{\nu}_{k,\infty}^2 B(\nu)^k$.
Since $d\leq X^{9/10}$ and $k\geq 100$, we in fact have
\begin{equation}
\label{INEQ:clean-variance-remainder-bound}
r_d\ll (X^{3-\delta} + X^{3+2\delta} K^{-2/3})\cdot \norm{\nu}_{k,\infty}^2 B(\nu)^k.
\end{equation}

But $g$ is multiplicative in $n$, and we have $g(n)\in (0,1)$ for all $n\geq 2$ (because, for instance, $F_0(\bm{0}) = 0$ and $F_0(1,0,0) = 1$).
And by definition, $P = \prod_{p<X^{\xi\hbar}} p$.
Directly if $\xi=0$,
and by the Selberg sieve \cite{iwaniec2004analytic}*{Theorem~6.4 and (6.80)} if $\xi=1$,
we conclude from \eqref{EQN:final-conditional-estimate-for-variance-restricted-to-dZ} that
\begin{equation}
\sum_{a\in \ZZ:\, \gcd(a,P) = 1}[N_{a,\nu}(X) - s_a(K)\sigma_{\infty,a,\nu}(X)]^2
\leq \frac{3! X^3 \norm{\nu}_{L^2(\RR^3)}^2}{H^\xi} + R(P),
\end{equation}
where
\begin{equation*}
\begin{split}
    H &\defeq \sum_{d<X^\hbar} \mu(d)^2 \prod_{p\mid d} \frac{g(p)}{1-g(p)}
    \geq 1, \\
    R(P) &\ll_\eps \sum_{d\leq X^{2\xi\hbar}} d^\eps \abs{r_d}
    \ll X^\eps (X^{3-\delta/10} + X^{3+2.9\delta} K^{-2/3}) \cdot \norm{\nu}_{k,\infty}^2 B(\nu)^k.
\end{split}
\end{equation*}
(We use \eqref{INEQ:clean-variance-remainder-bound} to bound $r_d$.)
The Hasse bound for elliptic curves gives $g(p) = p^{-1} + O(p^{-3/2})$;
so $\sum_{p\leq x} g(p) \log{p} = \log{x} + O(1)$ (for $x\geq 1$) and $\sum_{p>1} g(p)^2 \log{p} < \infty$.
By \cite{iwaniec2004analytic}*{\S6.6, derivation of (6.85) using Theorem~1.1}, then, $H \asymp \log(X^\hbar)$.
Theorem~\ref{THM:conditional-variance-evalulation-over-primes} follows.
\end{proof}

\section{Statistics of truncated singular series}
\label{SEC:statistics-of-s_a(K)}


For integers $K\geq 1$, recall the definition of $s_a(K)$ from \eqref{EQN:define-s_a(K)}.
In this section, we will prove that $s_a(K)$ is typically sizable (see Theorem~\ref{THM:s_a(K)-is-typically-sizable} below).
The proof makes use of an ``approximate inverse'' of $s_a(K) = \sum_{n\leq K} n^{-3} T_a(n)$.
There is some flexibility here; we let
\begin{equation}
\label{EQN:define-M_a(K)}
    M_a(K)\defeq \sum_{n\leq K:\, \gcd(n,30) = 1} \mu(n) n^{-3} T_a(n).
\end{equation}
%
%
The strategy for Theorem~\ref{THM:s_a(K)-is-typically-sizable} is to show that $s_a(K)M_a(K)$ is typically sizable, and that $M_a(K)$ is typically bounded.
For analytic convenience, let $T^\natural_a(n)\defeq n^{-2} T_a(n)$.

For integers $a$ and primes $p\nmid 3a$, Propositions~\ref{PROP:mod-p-bounds-on-T_a} and~\ref{PROP:T_a(p^l)-support-bound} below imply $T^\natural_a(p)\ll 1$ and $T_a(p^{\geq 2}) = 0$.
Indeed, if $a\neq 0$, then $s_a(K)$ resembles the value of an $L$-function at $1$,
and in this sense Theorem~\ref{THM:s_a(K)-is-typically-sizable} is related to work such as \cite{granville2003distribution}.
But there are also significant issues with this comparison, because $\abs{T^\natural_a(n)}$ can be as large as $n$ or so.
(Trivially $\abs{T^\natural_a(n)}\leq n^2$.)
The most serious issues arise when $a$ has a large cube divisor.

For an integer $n\geq 1$,
let $\map{sq}(n)\defeq \prod_{p^2\mid n} p^{v_p(n)}$ denote the \emph{square-full part} of $n$,
and $\map{cub}(n)\defeq \prod_{p^3\mid n} p^{v_p(n)}$ the \emph{cube-full part} of $n$.
Let $\mcal{S}(D)$ be the set of nonzero integers $a$ with $\map{sq}(\abs{a})\leq D$.
Let $\mcal{C}(D)$ be the set of nonzero integers $a$ with $\map{cub}(\abs{a})\leq D$.
We call an integer $n\geq 1$ \emph{square-full} if $n=\map{sq}(n)$,
and \emph{cube-full} if $n=\map{cub}(n)$.

We now collect some basic facts about $T_a(n)$.

\begin{lemma}
\label{LEM:T_a-bound-for-nearly-cube-free-n}
Let $a,n\in \ZZ$ with $n\geq 1$.
Let $n_3\defeq \map{cub}(n)$.
Then $\abs{T^\natural_a(n)}\leq O(1)^{\omega(n)} (n_3n)^{1/2}$.
\end{lemma}

\begin{proof}
By Proposition~\ref{PROP:standard-multiplicativity-properties-of-T_a,S^+_0}, it suffices to prove the lemma when $n=p^l$ for some prime $p$ and integer $l\geq 1$.
But by \eqref{EQN:define-T_a(m)} and \cite{vaughan1997hardy}*{(4.24) and Lemma~4.7}, we have
\begin{equation*}
    p^{3\floor{(l-1)/3}}\cdot p^{-3l}T_a(p^l)
    \ll p^{l-3/2}\bm{1}_{3\mid l-1} + p^{l-3}\bm{1}_{3\nmid l-1}.
\end{equation*}
Thus $T^\natural_a(p^l)
\ll p^{l-1/2}\bm{1}_{3\mid l-1}
+ p^{l-1}\bm{1}_{3\mid l-2}
+ p^{l}\bm{1}_{3\mid l}
\ll p^{l/2}\cdot \bm{1}_{l\in \set{1,2}} + p^l\cdot \bm{1}_{l\geq 3}$.
(There should be more robust proofs, but since $F_0$ is diagonal, it is convenient to call on \cite{vaughan1997hardy}.)
\end{proof}

For the proofs of the next two results, let $N_a(m)$ denote the number of solutions $\bm{y}\in (\ZZ/m\ZZ)^3$ to $F_0(\bm{y}) \equiv a\bmod{m}$.

\begin{proposition}
\label{PROP:mod-p-bounds-on-T_a}
Let $a\in \ZZ$.
Let $p$ be a prime.
Then the following hold:
\begin{enumerate}
    \item If $p\nmid 3a$, then $\abs{T^\natural_a(p)} \leq 6 + 2p^{-1/2}$.
    \item If $p\mid 3a$, then $\abs{T^\natural_a(p)} \leq 6p^{1/2}$.
    \item If $p\geq 7$, then $\abs{T^\natural_a(p)} < 0.99p$.
\end{enumerate}
\end{proposition}

\begin{proof}
By \eqref{EQN:define-T_a(m)}, we have
\begin{equation}
\label{EQN:T_a(p)-via-point-counts}
    T_a(p) = p N_a(p) - p^3.
\end{equation}

\emph{Case~1: $p\nmid 3a$.}
Then the equation $F_0(\bm{y}) = aw^3$ cuts out a smooth cubic hypersurface in $\PP^3$ over $\FF_p$.
The ``curve at infinity'' (cut out by $w=0$), namely $F_0(\bm{y}) = 0$ in $\PP^2$, is also smooth.
So by the Weil conjectures, $N_a(p) = (p^2+p+1 + \lambda_1) - (p+1 - \lambda_2)$, where $\lambda_1=\lambda_1(a;p)$ and $\lambda_2=\lambda_2(p)$ are integers with $\abs{\lambda_1}\leq 6p$ and $\abs{\lambda_2}\leq 2\sqrt{p}$.
Thus $\abs{T_a(p)}\leq 6p^2+2p^{3/2}$, and $\abs{T^\natural_a(p)}\leq 6+2p^{-1/2}$.
So (1)--(3) hold (with (2) being vacuous).

\emph{Case~2: $p=3$.}
Then $\abs{T^\natural_a(p)}\leq p^2$ trivially.
(In fact, $T_a(p) = 0$, but we do not need this special fact.)
So (1)--(3) hold (with (1) and (3) being vacuous).

\emph{Case~2: $p\mid 3a$ and $p\neq 3$.}
Then $p\mid a$.
So $T_a(p) = T_0(p)$.
But $F_0(\bm{y}) = 0$ in $\Aff^3$ is an affine cone over a curve.
This curve, $F_0(\bm{y}) = 0$ in $\PP^2$, is smooth.
So by the Weil conjectures, $N_0(p) = 1 + (p-1)(p+1-\lambda_2(p))$, where $\abs{\lambda_2(p)}\leq 2\sqrt{p}$.
Thus $\abs{T_a(p)}\leq 2p(p-1)\sqrt{p}$, and $\abs{T^\natural_a(p)}\leq 2(p^{1/2}-p^{-1/2})$.
So (1)--(3) hold (with (1) being vacuous).
\end{proof}

We now strengthen the vanishing result in \cite{vaughan1997hardy}*{Lemma~4.7}.

\begin{proposition}
\label{PROP:T_a(p^l)-support-bound}
Let $p$ be a prime.
Let $a,l\in \ZZ$ with $l\geq 2$.
If $p^{l-1}\nmid 3a$, then $T_a(p^l)=0$.
\end{proposition}

\begin{proof}
[Proof when $p\neq 3$]
By \eqref{EQN:define-T_a(m)}, we have
\begin{equation}
\label{EQN:T_a(p^l)-via-point-counts}
    T_a(p^l) = p^l N_a(p^l) - p^{l-1} p^3 N_a(p^{l-1}).
\end{equation}
Let $N^0_a(p^v)$ (for $v\geq 1$) denote the number of solutions $\bm{y}\in (\ZZ/p^v\ZZ)^3$ to $F_0(\bm{y}) \equiv a\bmod{p^v}$ with $p\mid \grad{F_0}(\bm{y})$.
(Here $\grad{F_0}(\bm{y}) = (3y_1^2,3y_2^2,3y_3^2)$.)
Then by \eqref{EQN:T_a(p^l)-via-point-counts} and Hensel's lemma,
\begin{equation}
\label{EQN:Hensel-simplify-T_a(p^l)-point-counts-general}
    T_a(p^l) = p^l N^0_a(p^l) - p^{l-1} p^3 N^0_a(p^{l-1}).
\end{equation}

\emph{Case~1: $p^3\nmid a$.}
Suppose $p^{l-1}\nmid 3a$.
Then $p^{\min(3, l-1)}\nmid a$,
so all solutions $\bm{y}\bmod{p^{l-1}}$ to $F_0(\bm{y})\equiv a\bmod{p^{l-1}}$ must be \emph{primitive} (i.e.~satisfy $p\nmid \bm{y}$), and thus satisfy $p\nmid \grad{F_0}(\bm{y})$.
Thus $N^0_a(p^v)=0$ for $v\geq l-1$, whence $T_a(p^l) = 0$ by \eqref{EQN:Hensel-simplify-T_a(p^l)-point-counts-general}.

\emph{Case~2: $p^3\mid a$.}
Suppose $p^{l-1}\nmid 3a$.
Then $l\geq 5$.
But $p\mid \grad{F_0}(\bm{y}) \Rightarrow p\mid \bm{y}$,
so $N^0_a(p^v) = p^6 N_{a/p^3}(p^{v-3})$ for $v\geq 4$,
whence $T_a(p^l) = p^9 T_{a/p^3}(p^{l-3})$ by \eqref{EQN:Hensel-simplify-T_a(p^l)-point-counts-general}.
Noting that $l-3\geq 2$ and $p^{l-4}\nmid 3a/p^3$, a recursive argument now yields $T_a(p^l)=0$.
\end{proof}

\begin{proof}
[Proof when $p=3$]
Both \eqref{EQN:T_a(p^l)-via-point-counts} and \eqref{EQN:Hensel-simplify-T_a(p^l)-point-counts-general} hold here, but \eqref{EQN:Hensel-simplify-T_a(p^l)-point-counts-general} is no longer simpler than \eqref{EQN:T_a(p^l)-via-point-counts}.
However, let $N^1_a(p^v)$ (for $v\geq 2$) denote the number of solutions $\bm{y}\in (\ZZ/p^v\ZZ)^3$ to $F_0(\bm{y}) \equiv a\bmod{p^v}$ with $p^2\mid \grad{F_0}(\bm{y})$.
A standard Hensel argument (based on writing $\bm{y} = \bm{y}_0 + p^{l-2}\bm{z}$, rather than $\bm{y} = \bm{y}_0 + p^{l-1}\bm{z}$) shows that if $l\geq 3$, then
\begin{equation}
\label{EQN:Hensel-simplify-T_a(3^l)-point-counts}
T_a(p^l) = p^l N^1_a(p^l) - p^{l-1} p^3 N^1_a(p^{l-1}).
\end{equation}

\emph{Case~1: $p^3\nmid a$.}
Suppose $p^{l-1}\nmid 3a$.
Then $l\geq 3$, and $p^{\min(3, l-2)}\nmid a$.
So all solutions $\bm{y}\bmod{p^{l-1}}$ to $F_0(\bm{y})\equiv a\bmod{p^{l-1}}$ are \emph{primitive}, and thus (since $l-1\geq 2$) satisfy $p^2\nmid \grad{F_0}(\bm{y})$.
Thus $N^1_a(p^v)=0$ for $v\geq l-1$, whence $T_a(p^l) = 0$ by \eqref{EQN:Hensel-simplify-T_a(3^l)-point-counts}.

\emph{Case~2: $p^3\mid a$.}
Suppose $p^{l-1}\nmid 3a$.
Then $l\geq 6$.
But $p^2\mid \grad{F_0}(\bm{y}) \Rightarrow p\mid \bm{y}$,
so $N^1_a(p^v) = p^6 N_{a/p^3}(p^{v-3})$ for $v\geq 5$,
whence $T_a(p^l) = p^9 T_{a/p^3}(p^{l-3})$ by \eqref{EQN:Hensel-simplify-T_a(3^l)-point-counts}.
Noting that $l-3\geq 3$ and $p^{l-4}\nmid 3a/p^3$,
a recursive argument yields $T_a(p^l)=0$.
\end{proof}

Now recall \eqref{EQN:define-s_a(K)} and \eqref{EQN:define-M_a(K)}.
For every $a\in \ZZ$, we have
\begin{equation}
\label{EQN:s_a(K)M_a(K)-expand}
    s_a(K) M_a(K)
    = \sum_{n\leq K} c_a(n)
    + \sum_{\substack{n_1,n_2\leq K: \\ n_1n_2>K,\; \gcd(n_2,30)=1}} n_1^{-3}T_a(n_1)\cdot \mu(n_2)n_2^{-3}T_a(n_2),
\end{equation}
where for integers $n\geq 1$ we let
\begin{equation*}
    c_a(n)\defeq n^{-3} \sum_{n_1n_2 = n:\, \gcd(n_2,30)=1} T_a(n_1)\cdot \mu(n_2) T_a(n_2).
\end{equation*}
By Proposition~\ref{PROP:standard-multiplicativity-properties-of-T_a,S^+_0}, $c_a(n)$ is a Dirichlet convolution of two multiplicative functions, so $c_a(n)$ is multiplicative in $n$.
For any prime $p$ and integer $l\geq 1$, we have
\begin{equation}
\label{EQN:c_a(p^l)-expand}
    c_a(p^l) = p^{-3l}[T_a(p^l) - T_a(p^{l-1})T_a(p)\bm{1}_{p\geq 7}]
    = p^{-l}[T^\natural_a(p^l) - T^\natural_a(p^{l-1})T^\natural_a(p)\bm{1}_{p\geq 7}].
\end{equation}


We now bound $c_a(n)$ using our work above (for convenience;
\cite{vaughan1997hardy}*{Lemma~4.7}, a ``stratification'' result for $p^{-l}T^\natural_a(p^l)$ based on $\gcd(p^l,a)$, might also suffice).

\begin{lemma}
\label{LEM:general-bound-on-c_a(n)}
Let $a,n\in \ZZ$ with $a\neq 0$ and $n\geq 1$.
Let $a_2\defeq \map{sq}(\abs{a})$ and $n_3\defeq \map{cub}(n)$.
Let $n_2\defeq \map{sq}(n/n_3)$ and $n_1\defeq n/n_2n_3$.
Then $n_2$ is a square, and
\begin{equation}
\label{INEQ:technical-multiplicative-form-of-general-bound-on-c_a(n)}
    \abs{c_a(n)} \le O(1)^{\omega(n)} n_1^{-2}
    \cdot n_2^{-1} \gcd(n_2^{1/2}, 3a)
    \cdot n_3^{-1/2} \prod_{p\mid n_3:\, v_p(n_3)\leq 2v_p(9a_2)} p^{v_p(n_3)/2}.
\end{equation}
\end{lemma}

\begin{proof}
Since $n/n_3$ is cube-free, $n_2 = \rad(n_2)^2$.
As for \eqref{INEQ:technical-multiplicative-form-of-general-bound-on-c_a(n)}, note that both sides of \eqref{INEQ:technical-multiplicative-form-of-general-bound-on-c_a(n)} are multiplicative in $n$, so
we may assume $n=p^l$, where $p$ is prime and $l\geq 1$ is an integer.

\emph{Case~1: $l=1$.}
By \eqref{EQN:c_a(p^l)-expand}, $c_a(p) = 0$ if $p\geq 7$, and $c_a(p)\ll p^{-2}$ trivially if $p\leq 5$.
So $c_a(n)\ll n^{-2}$.
This proves \eqref{INEQ:technical-multiplicative-form-of-general-bound-on-c_a(n)}, since $(n_1,n_2,n_3) = (n,1,1)$.

\emph{Case~2: $l=2$.}
Again, use \eqref{EQN:c_a(p^l)-expand}.
If $p\nmid 3a$, then Propositions~\ref{PROP:T_a(p^l)-support-bound} and~\ref{PROP:mod-p-bounds-on-T_a} give $c_a(p^2) = p^{-2}[0 + O(1)] \ll p^{-2}$.
If $p\mid 3a$, then Lemma~\ref{LEM:T_a-bound-for-nearly-cube-free-n} and Proposition~\ref{PROP:mod-p-bounds-on-T_a} give $c_a(p^2)\ll p^{-2}[p+O(p)]\ll p^{-1}$.
Either way, $c_a(p^2)\ll n^{-1}\gcd(n^{1/2},3a)$.
This proves \eqref{INEQ:technical-multiplicative-form-of-general-bound-on-c_a(n)}, since $(n_1,n_2,n_3) = (1,n,1)$.

\emph{Case~3: $l\geq 3$.}
If $p^{\max(2, l-2)}\nmid 3a$, then \eqref{EQN:c_a(p^l)-expand} and Proposition~\ref{PROP:T_a(p^l)-support-bound} (plus Lemma~\ref{LEM:T_a-bound-for-nearly-cube-free-n}
if $l=3$) give $c_a(p^l)
\ll p^{-l/2}$.
In general, \eqref{EQN:c_a(p^l)-expand} and Lemma~\ref{LEM:T_a-bound-for-nearly-cube-free-n} give $c_a(p^l)
\ll 1$.
Either way,
\begin{equation*}
    c_a(n)
    \ll n^{-1/2} (p^{v_p(n)/2})^{\bm{1}_{v_p(3a)\geq \max(2,l-2)}}
    \ll n^{-1/2} (p^{v_p(n)/2})^{\bm{1}_{v_p(9a_2)\geq l/2}},
\end{equation*}
where in the final step we have used the fact that $p^{\max(2,l-2)}$ is square-full, $\map{sq}(3a)\mid \map{sq}(9a_2)$, and $\max(2,l-2)\geq l/2$.
This proves \eqref{INEQ:technical-multiplicative-form-of-general-bound-on-c_a(n)}, since $(n_1,n_2,n_3) = (1,1,n)$.
\end{proof}

In general, the series $\sum_{n\leq K} c_a(n)$ behaves better than $s_a(K)$.

\begin{proposition}
\label{PROP:c_a(n)-series-to-product}
Let $a,D,K\in \ZZ$ with $a\neq 0$ and $D,K\geq 1$.
Suppose $a\not\equiv \pm4\bmod{9}$.
\begin{enumerate}
    \item For $p$ prime, $1+c_a(p)+c_a(p^2)+\cdots$ converges absolutely to a real number $\gamma_p(a)>0$.
    \item If $p\nmid 3a$, then $\gamma_p(a)\geq (1-p^{-2})^{O(1)}$.
    If $p\mid 3a$, then $\gamma_p(a)\geq (1-p^{-1})^{O(1)}$.
    \item $\prod_{\textnormal{$p$ prime}} \gamma_p(a)$ converges absolutely to a real number $\gamma(a)\gg \prod_{p\mid a} (1-p^{-1})^{O(1)}$.
    \item If $a\in \mcal{S}(D)$, then $\sum_{n\leq K} c_a(n) = \gamma(a) + O_\eps(\abs{a}^\eps K^{\eps-1/6} D)$.
\end{enumerate}
\end{proposition}

\begin{proof}
By \eqref{EQN:c_a(p^l)-expand}, Proposition~\ref{PROP:T_a(p^l)-support-bound}, and \eqref{EQN:define-F_a-p-adic-density}, the sum $\sum_{j\ge 0} c_a(p^j)$ converges \emph{absolutely} to
\begin{equation}
\label{EQN:gamma_p-via-sigma_p}
    \gamma_p(a)\defeq \sigma_{p,a}\cdot (1-p^{-3}T_a(p)\bm{1}_{p\geq 7}),
\end{equation}
since $a\ne 0$.
Here $\sigma_{p,a}>0$
by \cite{browning2021cubic}*{Lemma~2.16},
since $a\not\equiv \pm4\bmod{9}$.
And
\begin{equation}
\label{EQN:positive-mollifier/inverse-Euler-factor-at-p}
    1-p^{-3}T_a(p)\bm{1}_{p\geq 7} \in (0.01, 1.99)
\end{equation}
by Proposition~\ref{PROP:mod-p-bounds-on-T_a}.
So $\gamma_p(a)>0$ by \eqref{EQN:gamma_p-via-sigma_p}.
This completes the proof of (1).

Now we bound $\gamma_p(a)$ from below.
First suppose $p\leq 5$.
By \cite{browning2021cubic}*{Lemma~2.16}, $\sigma_{p,a} \ge p^{-6} \gg 1$.
So $\gamma_p(a)\gg 1$ by \eqref{EQN:gamma_p-via-sigma_p} and \eqref{EQN:positive-mollifier/inverse-Euler-factor-at-p}.
Now suppose $p\geq 7$.
Let $N^\ast_a(p)$ denote the number of solutions $\bm{y}\in \FF_p^3\setminus \set{\bm{0}}$ to $F_0(\bm{y})=a$.
By \eqref{EQN:define-F_a-p-adic-density} and Hensel's lemma, $\sigma_{p,a}\geq p^{-2}N^\ast_a(p)$.
But $N^\ast_a(p)\geq N_a(p)-1$,
and $N_a(p) = p^2 + p^{-1}T_a(p)$ by \eqref{EQN:T_a(p)-via-point-counts},
so $\sigma_{p,a}\geq 1+p^{-3}T_a(p)-p^{-2}$.
Since $p\geq 7$, it follows from \eqref{EQN:gamma_p-via-sigma_p}, \eqref{EQN:positive-mollifier/inverse-Euler-factor-at-p}, and Proposition~\ref{PROP:mod-p-bounds-on-T_a} that
$
\gamma_p(a) \geq 1-p^{-6}T_a(p)^2 - 1.99p^{-2}
= 1 - O(p^{-2}\bm{1}_{p\nmid 3a} + p^{-1}\bm{1}_{p\mid 3a})$.
Since $\gamma_p(a)>0$ by (1), we conclude that (2) holds for $p\geq 7$.
By enlarging the $O(1)$'s in (2) if necessary, we can then ensure that (2) holds for all primes $p$.

For (3), first note that if $p\geq 7$ and $p\nmid 3a$, then $N^\ast_a(p) = N_a(p)$, so $\sigma_{p,a} = p^{-2}N_a(p)$ by \eqref{EQN:define-F_a-p-adic-density} and Hensel's lemma;
and hence the arguments in the previous paragraph imply $\gamma_p(a) = 1 + O(p^{-2})$ (uniformly over $p\nmid 30a$).
Thus $\prod_{\textnormal{$p$ prime}} \gamma_p(a)$ converges absolutely to some $\gamma(a)\in \RR_{>0}$.
By (2), then, $\gamma(a)\gg \prod_{p\mid 3a} (1-p^{-1})^{O(1)}$.
So (3) holds.

For (4), note that by Lemma~\ref{LEM:general-bound-on-c_a(n)} (and the definition of $\mcal{S}(D)$),
\begin{equation}
\label{INEQ:apply-c_a(n)-bound-pointwise-and-over-extend}
    \sum_{n>K} \abs{c_a(n)}
    \ll_\eps \sum_{\substack{n_1,m_2,n_3\geq 1: \\ n_1m_2^2n_3\geq K,\; n_3=\map{cub}(n_3)}}
    (n_1m_2^2n_3)^\eps n_1^{-2} m_2^{-2} \gcd(m_2,3a) n_3^{-1/2} (9D).
\end{equation}
The contribution to the right-hand side of \eqref{INEQ:apply-c_a(n)-bound-pointwise-and-over-extend} from $(n_1,m_2,n_3)\in [N_1,2N_1) \times [M_2,2M_2) \times [N_3,2N_3)$ (where $(2N_1)(2M_2)^2(2N_3)\geq K$) is at most $O_\eps((N_1M_2N_3)^{2\eps} D)$ times
\begin{equation*}
\frac{N_1 N_3^{1/3}}{N_1^2 M_2^2 N_3^{1/2}}
\sum_{d\mid 3a} \sum_{m_2\in [M_2,2M_2):\, d\mid m_2} d
\ll \sum_{d\mid 3a} \frac{N_1 M_2 N_3^{1/3}}{N_1^2 M_2^2 N_3^{1/2}}
\ll_\eps \frac{\abs{a}^\eps}{(N_1M_2^2N_3)^{1/6}}
\ll_\eps \frac{\abs{a}^\eps K^{3\eps-1/6}}{(N_1M_2^2N_3)^{3\eps}}.
\end{equation*}
Summing over $N_1,M_2,N_3\in \set{1,2,4,8,\ldots}$, 
we get $\sum_{n>K} \abs{c_a(n)} \ll_\eps \abs{a}^\eps K^{3\eps-1/6} D$.
Thus $\sum_{n\geq 1} c_a(n)$ converges absolutely to $\gamma(a)$;
and furthermore, (4) holds.
\end{proof}

To bound $M_a(K)$ on average, and to handle the sum over $n_1n_2>K$ in \eqref{EQN:s_a(K)M_a(K)-expand}, we can use Lemma~\ref{LEM:restricted-2jth-moment-bound-for-general-convolution} below, whose proof requires the following two lemmas.


\begin{lemma}
\label{LEM:T_a-support-bound-for-nearly-square-free-a}
Let $D,n\geq 1$ be integers.
Let $a\in \mcal{S}(D)$.
If $T_a(n)\neq 0$, then $n\in \mcal{C}(27D^{3/2})$.
\end{lemma}

\begin{proof}
Suppose $T_a(n)\neq 0$.
Let $a_2\leq D$ be the square-full part of $\abs{a}$.
Suppose $p$ is a prime dividing the cube-full part $n_3$ of $n$.
Since $T_a(n)\neq 0$, Proposition~\ref{PROP:standard-multiplicativity-properties-of-T_a,S^+_0} implies $T_a(p^{v_p(n)})\neq 0$.
So by Proposition~\ref{PROP:T_a(p^l)-support-bound}, $p^{v_p(n)-1}\mid 3a$.
Here $v_p(n)\ge 3$; so $p^{v_p(n)-1}\mid \map{sq}(3a)\mid \map{sq}(9a_2)$.
Therefore, $9a_2$ is divisible by $\prod_{p\mid n_3} p^{v_p(n)-1} \geq n_3^{2/3}$,
whence $n_3\leq (9a_2)^{3/2}\leq 27D^{3/2}$.
\end{proof}

The next lemma controls the large values of $\abs{T^\natural_a(n)}$ on average over $a$.
Let
\begin{equation*}
T^\natural_b(\bm{m})\defeq T^\natural_b(m_1)\cdots T^\natural_b(m_r),
\quad T^\natural_b(\bm{n})\defeq T^\natural_b(n_1)\cdots T^\natural_b(n_r),
\end{equation*}
whenever the variables $b$, $r$, $m_1,\dots,m_r$, $n_1,\dots,n_r$ are clear from context.

\begin{lemma}
\label{LEM:technical-complete-T_b-moment-bound}
Let $r,m_1,n_1,\dots,m_r,n_r\geq 1$ be integers.
Let $m_{i,3}$, $n_{i,3}$ denote the cube-full parts of $m_i$, $n_i$, respectively.
Suppose $m_1\cdots m_rn_1\cdots n_r$ is square-full.
Then
\begin{equation*}
    \EE_{b\in \ZZ/m_1\cdots m_rn_1\cdots n_r\ZZ}[\abs{T^\natural_b(\bm{m})}\cdot \abs{T^\natural_b(\bm{n})}]
    \ll_{r,\eps} \frac{\prod_{i=1}^{r}[(m_{i,3}n_{i,3})^{1/2} (m_in_i)^{1/2+\eps}]}{\rad(m_1\cdots m_rn_1\cdots n_r)}.
\end{equation*}
\end{lemma}

\begin{proof}
By Proposition~\ref{PROP:standard-multiplicativity-properties-of-T_a,S^+_0}, we may assume that $m_1\cdots m_rn_1\cdots n_r$ is a prime power;
or equivalently, that $\rad(m_1\cdots m_rn_1\cdots n_r)$ equals some prime $p$.
Applying Lemma~\ref{LEM:T_a-bound-for-nearly-cube-free-n} to residues $b$ with $p\mid 3b$,
and Propositions~\ref{PROP:mod-p-bounds-on-T_a} and~\ref{PROP:T_a(p^l)-support-bound} to residues $b$ with $p\nmid 3b$, gives
\begin{equation*}
    \EE_{b\in \ZZ/m_1\cdots m_rn_1\cdots n_r\ZZ}[\abs{T^\natural_b(\bm{m})}\cdot \abs{T^\natural_b(\bm{n})}]
    \ll_r \frac{\prod_{i=1}^{r}[(m_{i,3}n_{i,3})^{1/2} (m_in_i)^{1/2}]}{p/\gcd(3,p)} + 1.
\end{equation*}
But $m_1\cdots m_rn_1\cdots n_r$ is square-full,
so $\prod_{i=1}^{r} (m_in_i)^{1/2} \geq p$, and the lemma follows.
\end{proof}




The following result can loosely be thought of as an algebro-geometric analog of \cite{duke2000problem}*{Theorem~4},
emphasizing a different aspect.
It only applies over relatively small moduli.

\begin{lemma}
\label{LEM:restricted-2jth-moment-bound-for-general-convolution}
Let $A,K_1,K_2,D,j\geq 1$ be integers.
Let $\beta\maps \ZZ^2 \to \RR$ be a function supported on $[K_1,2K_1) \times [K_2,2K_2)$.
For $a\in \ZZ$, let
\begin{equation}
\label{EQN:define-general-T_a-convolution-type-polynomial-P_a}
P_a\defeq \sum_{m,n\geq 1:\, n\in \mcal{S}(1)} \beta(m,n) \cdot T^\natural_a(m)T^\natural_a(n).
\end{equation}
Suppose $A\geq (K_1K_2)^{3j}$.
Then
\begin{equation}
\label{INEQ:restricted-2jth-moment-bound-for-general-convolution}
\sum_{a\in [-A,A]\cap \mcal{S}(D)} \abs{P_a}^{2j}
\ll_{j,\eps} A \cdot \min(K_1,D^{3/2})^j\cdot (K_1K_2)^{j+\eps} \cdot (\max{\abs{\beta}})^{2j}.
\end{equation}
\end{lemma}

\begin{proof}
Let $K\defeq K_1K_2$ and $r\defeq 2j$.
For $a\in \ZZ$, let
\begin{equation*}
Q_a\defeq \sum_{m,n\ge 1:\, m\in \mcal{C}(27D^{3/2}),\; n\in \mcal{S}(1)}
\beta(m,n) \cdot T^\natural_a(m)T^\natural_a(n).
\end{equation*}
By \eqref{EQN:define-general-T_a-convolution-type-polynomial-P_a} and Lemma~\ref{LEM:T_a-support-bound-for-nearly-square-free-a}, we have $P_a = Q_a$ for all $a\in \mcal{S}(D)$.
Thus we may replace each $P_a$ in \eqref{INEQ:restricted-2jth-moment-bound-for-general-convolution} with $Q_a$;
and by positivity, we may then extend $[-A,A]\cap \mcal{S}(D)$ to the full interval $[-A,A]$.
It follows that the left-hand side of \eqref{INEQ:restricted-2jth-moment-bound-for-general-convolution} is at most
\begin{equation}
\label{EQN:over-extend-plus-Poisson}
\sum_{a\in [-A,A]} \abs{Q_a}^r
\ll \sum_{\substack{m_1,\dots,m_r\in \ZZ_{\geq 1}\cap \mcal{C}(27D^{3/2}), \\
n_1,\dots,n_r\in \ZZ_{\geq 1}\cap \mcal{S}(1), \\
b\in \ZZ/m_1\cdots m_rn_1\cdots n_r\ZZ}}
\beta(\bm{m},\bm{n}) T^\natural_b(\bm{m}) T^\natural_b(\bm{n})
\cdot \frac{A\cdot (1+O_k(K^{-k}))}{m_1\cdots m_rn_1\cdots n_r},
\end{equation}
where $\beta(\bm{m},\bm{n})\defeq \beta(m_1,n_1)\cdots \beta(m_r,n_r)$;
the inequality \eqref{EQN:over-extend-plus-Poisson} follows from Poisson summation after replacing (by positivity again) the sum over $a\in [-A,A]$ with a smoothed sum over $a\in \ZZ$.
The total contribution from $O_k(K^{-k})$ to the right-hand side of \eqref{EQN:over-extend-plus-Poisson} is trivially $\ll K^{2r}\cdot (\max{\abs{\beta}})^r\cdot K^r\cdot A\cdot O_k(K^{-k})$, which is satisfactory if $k\defeq 3r$, say.

It remains to handle the ``main'' term in \eqref{EQN:over-extend-plus-Poisson}.
Given $m_1,n_1,\dots,m_r,n_r\geq 1$, we have
\begin{equation*}
    \sum_{b\in \ZZ/m_1\cdots m_rn_1\cdots n_r\ZZ} T^\natural_b(\bm{m})T^\natural_b(\bm{n}) = 0
    \quad\textnormal{if $m_1\cdots m_rn_1\cdots n_r$ is not square-full}.
\end{equation*}
(This follows from Proposition~\ref{PROP:standard-multiplicativity-properties-of-T_a,S^+_0} and the fact that $\sum_{a\in \ZZ/p\ZZ} T_a(p) = 0$ for every prime $p$.
Note that we used a similar idea to prove Lemma~\ref{LEM:unbalanced-T_b-moment-vanishing}.)
On the other hand, if $m_1\cdots m_rn_1\cdots n_r$ is square-full, then Lemma~\ref{LEM:technical-complete-T_b-moment-bound} delivers the bound
\begin{equation*}
\EE_{b\in \ZZ/m_1\cdots m_rn_1\cdots n_r\ZZ}[\abs{T^\natural_b(\bm{m})}\cdot \abs{T^\natural_b(\bm{n})}]
\ll_{r,\eps} \frac{\min(K_1,D^{3/2})^{r/2}\cdot K^{r/2+\eps}}{\rad(m_1\cdots m_rn_1\cdots n_r)},
\end{equation*}
provided that we have $m_1,\dots,m_r\in \mcal{C}(27D^{3/2})$ and $n_1,\dots,n_r\in \mcal{S}(1)$.
It follows that
\begin{equation}
\label{INEQ:key-Poisson-main-term-bound}
\sum_{\substack{m_1,\dots,m_r\in \ZZ_{\geq 1}\cap \mcal{C}(27D^{3/2}), \\
n_1,\dots,n_r\in \ZZ_{\geq 1}\cap \mcal{S}(1), \\
b\in \ZZ/m_1\cdots m_rn_1\cdots n_r\ZZ}} \frac{\beta(\bm{m},\bm{n}) T^\natural_b(\bm{m}) T^\natural_b(\bm{n})}{m_1\cdots m_rn_1\cdots n_r}
\ll \sum_{R\leq (4K)^{r/2}} \sum_{\substack{m_1,\dots,m_r\mid R^\infty, \\
n_1,\dots,n_r\mid R^\infty}} \frac{C(\bm{m},\bm{n})}{R},
\end{equation}
where $C(\bm{m},\bm{n})\defeq \abs{\beta(\bm{m},\bm{n})}\cdot \min(K_1,D^{3/2})^{r/2}\cdot K^{r/2+\eps}$.
(For $u, v\in \ZZ$, we write $u\mid v^\infty$ if there exists a positive integer $k$ with $u\mid v^k$.)
But for any integers $R,N\geq 1$,
we have $\sum_{n\mid R^\infty} \bm{1}_{n\in [N,2N)}
\ll_\eps \sum_{n\mid R^\infty} (N/n)^\eps
= N^\eps \prod_{p\mid R} (1 - p^{-\eps})^{-1}
\ll_\eps N^\eps R^\eps$; so
\begin{equation}
\label{INEQ:epsilon-sharp-control-of-the-radical}
\sum_{m_1,\dots,m_r,n_1,\dots,n_r\mid R^\infty} \abs{\beta(\bm{m},\bm{n})}
\ll_{r,\eps} (\max{\abs{\beta}})^r (K_1 R)^{r\eps} (K_2 R)^{r\eps}.
\end{equation}
By \eqref{INEQ:epsilon-sharp-control-of-the-radical} and the bound $\sum_{R\le (4K)^{r/2}} R^{2r\eps-1} \ll_{r,\eps} K^{r^2\eps}$, the right-hand side of \eqref{INEQ:key-Poisson-main-term-bound} is $\ll_{r,\eps} (\max{\abs{\beta}})^r \min(K_1,D^{3/2})^{r/2} K^{r/2+(1+r+r^2)\eps}$.
Plugging this into \eqref{EQN:over-extend-plus-Poisson} gives \eqref{INEQ:restricted-2jth-moment-bound-for-general-convolution}.
\end{proof}

For integers $K\ge 1$ and reals $\eta>0$, let
\begin{equation}
\label{EQN:define-admissible-integers-with-eta-bounded-s_a(K)}
\mscr{E}(K;\eta)
\defeq \set{a\in \ZZ: a\not\equiv\pm4\bmod{9}}
\cap \set{a\in \ZZ: \abs{s_a(K)}\leq \eta}.
\end{equation}

\begin{theorem}
\label{THM:s_a(K)-is-typically-sizable}
Let $A,K,j\geq 1$ be integers.
Let $\eps,\eta>0$ be reals.
If $A\geq K^{6j}$, then
\begin{equation*}
\frac{\card{\mscr{E}(K;\eta) \cap [-A,A]}}{A}
\ll_{j,\eps} \eta^{1.8j} + A^\eps K^{-1.5j} + \eta^{-0.2j} K^{-0.8j} + K^{-1/24}.
\end{equation*}
\end{theorem}

\begin{proof}
By Proposition~\ref{PROP:c_a(n)-series-to-product},
there exists a constant $C=C(\eps)>0$ such that every element of $\mscr{E}(K;\eta)\cap [-A,A]\cap \mcal{S}(D)$ lies in one of the following sets:
\begin{enumerate}
    \item $\mscr{E}_1\defeq \set{a\in \mcal{S}(D): \abs{M_a(K)}\geq \eta^{-9/10}}$.
    
    \item $\mscr{E}_2\defeq \set{a\in \mcal{S}(D): \prod_{p\mid a}(1-p^{-1})^C \leq C\cdot (\eta^{1/10} + A^\eps K^{\eps-1/6}D)}$.
    
    \item $\mscr{E}_3\defeq \set{a\in \mcal{S}(D): \abs{s_a(K)M_a(K)-\sum_{n\leq K}c_a(n)}\geq \eta^{1/10}}$.
\end{enumerate}
Indeed, if $C$ is sufficiently large and $a\in \mscr{E}(K;\eta)\cap [-A,A]\cap \mcal{S}(D)\setminus (\mscr{E}_1\cup \mscr{E}_2)$,
then $\abs{s_a(K)}\leq \eta$ (since $a\in \mscr{E}(K;\eta)$)
and $\abs{M_a(K)}\leq \eta^{-9/10}$ (since $n\notin \mscr{E}_1$),
and $\sum_{n\leq K} c_a(n) \geq 2\eta^{1/10}$ (by Proposition~\ref{PROP:c_a(n)-series-to-product}, since $n\notin \mscr{E}_2$),
so $a\in \mscr{E}_3$.

We now bound $\mscr{E}_1$, $\mscr{E}_2$, $\mscr{E}_3$.
By \eqref{EQN:define-M_a(K)} and Lemma~\ref{LEM:restricted-2jth-moment-bound-for-general-convolution} (with $K_1=1$ and $1\leq K_2\leq K$),
\begin{equation}
\eta^{-1.8j}\cdot \card{\mscr{E}_1\cap [-A,A]}
\le \sum_{a\in [-A,A]\cap \mcal{S}(D)} \abs{M_a(K)}^{2j}
\ll_j A,
\end{equation}
provided $A\geq K^{3j}$.
By \eqref{EQN:s_a(K)M_a(K)-expand} and Lemma~\ref{LEM:restricted-2jth-moment-bound-for-general-convolution} (with $1\leq K_2,K_2\leq K$ and $(2K_1)(2K_2)\geq K$),
\begin{equation}
\eta^{0.2j}\cdot \card{\mscr{E}_3\cap [-A,A]}
\le \sum_{a\in [-A,A]\cap \mcal{S}(D)} \Bigl\lvert{s_a(K)M_a(K)-\sum_{n\leq K}c_a(n)}\Bigr\rvert^{2j}
\ll_{j,\eps} A D^{3j/2} K^{-j+\eps},
\end{equation}
provided $A\geq K^{6j}$.
And $\prod_{p\mid a}(1-p^{-1}) = \phi(\abs{a})/\abs{a}$, so
\begin{equation}
    C^{-18j} (\eta^{1/10} + A^\eps K^{\eps-1/6}D)^{-18j}
    \cdot \card{\mscr{E}_2\cap [-A,A]}
    \le \sum_{a\in [-A,A]\setminus \set{0}} \left(\frac{\abs{a}}{\phi(\abs{a})}\right)^{18Cj}
    \ll_{Cj} A
\end{equation}
by \cite{montgomery2007multiplicative}*{p.~61, (2.32)}.
Since $[-A,A]\setminus \mcal{S}(D)$ has size $\ll D^{-1/2}A$, we conclude that
\begin{equation*}
\frac{\card{\mscr{E}(K;\eta) \cap [-A,A]}}{A} \ll_{j,\eps}
\eta^{1.8j} + (A^\eps K^{\eps-1/6}D)^{18j} + \eta^{-0.2j}D^{3j/2}K^{-j+\eps} + D^{-1/2},
\end{equation*}
provided $A\geq K^{6j}$.
Taking $D=\floor{K^{1/12}}$ gives Theorem~\ref{THM:s_a(K)-is-typically-sizable}.
\end{proof}

\section{Applying increasingly cuspidal weights}
\label{SEC:apply-variance-estimates}

To prove Theorems~\ref{THM:main-result-on-integer-values} and~\ref{THM:main-result-on-prime-values}, we will combine Theorems~\ref{THM:unconditional-variance-evalulation-over-integers}, \ref{THM:conditional-variance-evalulation-over-primes}, and~\ref{THM:s_a(K)-is-typically-sizable}.
To apply Theorems~\ref{THM:unconditional-variance-evalulation-over-integers} and~\ref{THM:conditional-variance-evalulation-over-primes}, we need to choose a suitable weight $\nu$.
Fix a function $w_0\in C^\infty_c(\RR)$ with $w_0\geq 0$ everywhere, and $w_0\geq 1$ on $[-2,2]$.
Fix a function $w_2\in C^\infty_c(\RR)$ with $w_2\geq 0$ everywhere, $w_2\geq 1$ on $[1,10]$, and $\Supp{w_2}\belongs \RR_{>0}$.
Given a real $R\geq 2$, set
\begin{equation}
\label{EQN:define-key-weight-nu^star}
\nu^\star(\bm{y})
\defeq w_0(F_0(\bm{y}))
\int_{r\in [1,R]} d^\times{r}
\prod_{1\leq l\leq 3} w_2(\abs{y_l}/r)
\prod_{1\leq i<j\leq 3} w_2(\abs{y_i+y_j}/r),
\end{equation}
where $d^\times{r}\defeq dr/r$.
Clearly $\nu^\star\in C^\infty_c(\RR^3)$, and $\nu^\star$ satisfies \eqref{EQN:condition-for-very-clean} and \eqref{EQN:condition-for-symmetric}.

Let us consider what happens as we vary $R$.
It is clear from \eqref{EQN:define-key-weight-nu^star} that
\begin{equation}
\label{INEQ:bound-diameter-of-tentacled-nu}
\Supp{\nu^\star} \belongs \set{\bm{y}\in \RR^3: 1\ll \abs{y_1},\abs{y_2},\abs{y_3}\ll R};
\end{equation}
in particular, $B(\nu^\star) \ll R$.
On the other hand (reflecting the tentacled nature of $\Supp{\nu^\star}$),
\begin{equation}
\label{INEQ:bound-volume-of-tentacled-nu}
    \vol(\Supp\nu^\star)
    \ll \int_{r\in [1,R]} d^\times{r} \int_{y_1,y_2\in \RR} dy_1\,dy_2\,
    \bm{1}_{\abs{y_1},\abs{y_2}\in r\cdot \Supp{w_2}} \cdot r^{-2}
    \ll \log{R},
\end{equation}
because for any $r\in [1,R]$ and $y_1,y_2\in \pm r\cdot \Supp{w_2}$, the set $\set{y_3\in \pm r\cdot \Supp{w_2}: F_0(\bm{y})\in \Supp{w_0}}$ has measure $\ll r^{-2}$.
We also have the following key lemma.

\begin{lemma}
\label{LEM:log-growth-of-real-densities-for-constructed-nu^star}
Let $a\in \RR$ with $\abs{a}\leq X^3$.
Then $\sigma_{\infty,a,\nu^\star}(X)\gg \log{R}$.
\end{lemma}

\begin{proof}
Let $\tilde{a}\defeq a/X^3$.
Given $(y_2,y_3)$, let $y_1\defeq (\tilde{a}-y_2^3-y_3^3)^{1/3}$.
For $r\geq 1$, let $D_r\defeq [3.99r,4r]^2$.
Then for all $(y_2,y_3)\in D_r$, we have $F_0(\bm{y})=\tilde{a}\in [-1,1]$ and $y_1\in [-6r,-5r]$, and thus
\begin{equation*}
w_0(F_0(\bm{y})) \prod_{1\leq l\leq 3} w_2(\abs{y_l}/r) \prod_{1\leq i<j\leq 3} w_2(\abs{y_i+y_j}/r)
\geq 1.
\end{equation*}
By \eqref{EQN:surface-integral-representation-of-real-density}, \eqref{EQN:define-key-weight-nu^star}, and the nonnegativity of $w_0$, $w_2$, it follows that
\begin{equation*}
    \sigma_{\infty,a,\nu^\star}(X)
    = \int_{\RR^2} dy_2\,dy_3\,\nu^\star(\bm{y})\cdot (3y_1^2)^{-1}
    \gg \int_{r\in [1,R]} d^\times{r}\,\vol(D_r)\cdot r^{-2}
    \gg \log{R},
\end{equation*}
since $\vol(D_r)\gg r^2$.
\end{proof}

We also need some control on the norms \eqref{EQN:define-Sobolev-norm} of $\nu^\star$.
Since $w_0$, $w_2$ are fixed, we have
\begin{equation}
\label{INEQ:bound-max-modulus-of-tentacled-nu}
\nu^\star(\bm{y})
\ll \int_{\RR_{>0}} d^\times{r}\,w_2(\abs{y_1}/r)
= \bm{1}_{y_1\ne 0} \cdot \int_{\RR_{>0}} d^\times{r}\,w_2(1/r)
\ll 1,
\end{equation}
uniformly over $\bm{y}\in \RR^3$ and $R\geq 2$.
In general, for integers $k\geq 0$, we have\footnote{In \cite{wang2022thesis}*{Remark~2.2.15} the derivatives of $\nu$ are incorrectly stated to be $O_k(1)$.
This mistake does not affect the proof of \cite{wang2022thesis}*{Theorem~2.1.8}.
In any case, our present work is independent of \cite{wang2022thesis}.}
\begin{equation}
\label{INEQ:bound-derivatives-of-tentacled-nu}
\norm{\nu^\star}_{k,\infty} \ll_k B(\nu^\star)^{2k} \ll_k R^{2k};
\end{equation}
to see why, note that by the chain rule, any $y_i$-derivative of $w_0$ in \eqref{EQN:define-key-weight-nu^star} introduces a factor of $3y_i^2\ll B(\nu^\star)^2$, while any $y_i$-derivative of $w_2$ in \eqref{EQN:define-key-weight-nu^star} only introduces a factor of $1/r \ll 1$.

We are finally prepared to prove our main theorems, by adapting Chebyshev's inequality to approximate variances like \eqref{EQN:define-approximate-variance} (after removing $a$'s for which $s_a(K)$ is small).

\begin{proof}
[Proof of Theorems~\ref{THM:main-result-on-integer-values} and~\ref{THM:main-result-on-prime-values}]
We gradually increase our hypotheses.
First assume that \eqref{EQN:soft-HLH-general-homogeneous-weight} for $d=1$ holds for all clean functions $w\in C^\infty_c(\RR^6)$.
Plugging this and \eqref{EQN:linear-space-inclusion-exclusion}, \eqref{EQN:special-solution-Poisson-summation} into Theorem~\ref{THM:unconditional-variance-evalulation-over-integers} (with $d=1$), we get that for integers $X,K\geq 1$ with $K\leq X^{9/10}$, we have
\begin{equation}
\frac{\map{Var}(X,K;1)}{X^3}
\ll \norm{\nu^\star}_{L^2(\RR^3)}^2 + o_{\nu^\star;X\to\infty}(1)
+ K^{-1/2} \norm{\nu^\star}_{100,\infty}^2 B(\nu^\star)^{5000},
\end{equation}
where $o_{\nu^\star;X\to\infty}(1)$ denotes a quantity
that tends to $0$ as $X\to\infty$ (for any fixed $\nu^\star$).
Here $\norm{\nu^\star}_{L^2(\RR^3)}^2\ll \log{R}$ by \eqref{INEQ:bound-volume-of-tentacled-nu}--\eqref{INEQ:bound-max-modulus-of-tentacled-nu}, and $\norm{\nu^\star}_{100,\infty}^2 B(\nu^\star)^{5000}\ll R^{5400}$ by \eqref{INEQ:bound-derivatives-of-tentacled-nu}.
On the other hand, by Theorem~\ref{THM:s_a(K)-is-typically-sizable} with $\eta = (\log{R})^{-10/j}$ (for an integer $j\geq 1$), we have
\begin{equation}
    \frac{\card{\mscr{E}(K;\eta)\cap [-A,A]}}{A}
    \ll_{j,\eps} (\log{R})^{-18} + A^\eps K^{-1.5j} + (\log{R})^{2} K^{-0.8j} + K^{-1/24}
\end{equation}
for integers $A\geq K^{6j}$.
By \eqref{EQN:define-approximate-variance}, \eqref{EQN:define-admissible-integers-with-eta-bounded-s_a(K)}, and Lemma~\ref{LEM:log-growth-of-real-densities-for-constructed-nu^star},
we conclude (by letting $A,X\to \infty$ with $A\in [X^3/2, X^3]$, taking $K=\floor{A^{1/6j}}$, taking $\eps=1/6$, and taking $R\leq K^{1/20000}$) that
\begin{equation}
\begin{split}
\label{INEQ:endgame-splitting-template}
\frac{\card{\mcal{E}\cap [-A,A]}}{A}
&\ll \frac{\card{\mscr{E}(K;\eta)\cap [-A,A]}}{A}
+ \frac{\map{Var}(X,K;1)/X^3}{\eta^2 (\log{R})^2} \\
&\ll_j (\log{R})^{-18}
+ \frac{(\log{R}) + o_{R;A\to\infty}(1)}{(\log{R})^{2-20/j}}
\end{split}
\end{equation}
as $A\to\infty$.
Taking $j=40$ and $R\to\infty$ proves the first part of Theorem~\ref{THM:main-result-on-integer-values}.

What remains is similar.
Fix $(\delta, k)\in \RR_{>0}\times \ZZ_{\geq 1}$, and assume \eqref{EQN:hard-HLH-clean-weight-level-d} for $\xi=0$.
Note that \eqref{EQN:hard-HLH-clean-weight-level-d} remains true if we decrease $\delta$ or increase $k$; so we may assume $\delta\in (0,9/10)$ and $k\ge 5000$.
In fact, it will be convenient to assume $\delta=1/(12j)$ where $j\ge 1$ is a large integer.
Theorem~\ref{THM:conditional-variance-evalulation-over-primes} for $\xi=0$ now implies (assuming $X,K\ge 2$ and $K\le X^{9/10-\delta}$)
\begin{equation}
\frac{\map{Var}(X,K;1)}{X^3}
\ll_j \norm{\nu^\star}_{L^2(\RR^3)}^2
+ \frac{\norm{\nu^\star}_{k,\infty}^2 B(\nu^\star)^k}{\min(X^{\delta/11}, K^{2/3}X^{-3\delta})}.
\end{equation}
Let $A, X \to \infty$ with $A \in [X^3/2, X^3]$;
let $(K, \eps, R, \eta) = (\floor{A^{2\delta}}, 1/6, X^{\delta/300k}, (\log{R})^{-10/j})$.
Then $K = \floor{A^{1/6j}}$, so $A\geq K^{6j}$.
Applying \eqref{INEQ:bound-volume-of-tentacled-nu}--\eqref{INEQ:bound-derivatives-of-tentacled-nu}, Theorem~\ref{THM:s_a(K)-is-typically-sizable}, and Lemma~\ref{LEM:log-growth-of-real-densities-for-constructed-nu^star} as before,
we get by the first line of \eqref{INEQ:endgame-splitting-template} that
\begin{equation}
    \frac{\card{\mcal{E}\cap [-A,A]}}{A}
    \ll_j (\log{R})^{-18}
    + \frac{(\log{R}) + R^{5k} X^{-\delta/11}}{(\log{R})^{2-20/j}}
    \ll_j \frac{1}{(\log{R})^{1-20/j}},
\end{equation}
since $R^{5k} = X^{\delta/60}$.
Taking $j\to \infty$ proves the second part of Theorem~\ref{THM:main-result-on-integer-values}.

Finally, fix $(\delta, k)\in \RR_{>0}\times \ZZ_{\geq 1}$, and assume \eqref{EQN:hard-HLH-clean-weight-level-d} for $\xi=1$.
As in the previous paragraph, we assume $\delta=1/(12j)$ and $k\ge 5000$, where $j\in \ZZ_{\ge 1}$,
and let $(K, \eps, R, \eta) = (\floor{A^{2\delta}}, 1/6, X^{\delta/300k}, (\log{R})^{-10/j})$, where $A \in [X^3/2, X^3]$ and $A, X \to \infty$.
Theorem~\ref{THM:conditional-variance-evalulation-over-primes} for $\xi=1$,
when combined with \eqref{INEQ:bound-volume-of-tentacled-nu}--\eqref{INEQ:bound-derivatives-of-tentacled-nu}, Theorem~\ref{THM:s_a(K)-is-typically-sizable}, and Lemma~\ref{LEM:log-growth-of-real-densities-for-constructed-nu^star} as before,
then gives
\begin{equation}
    \frac{\card{\mcal{E}\cap \set{p\le A}}}{A}
    \ll_j (\log{R})^{-18}
    + \frac{(\log{R}/\log{X}) + R^{5k} X^{-\delta/11}}{(\log{R})^{2-20/j}}
    \ll_j \frac{1}{(\log{R})^{2-20/j}}.
\end{equation}
Taking $j\to \infty$ proves Theorem~\ref{THM:main-result-on-prime-values}.
\end{proof}


\section{Nonnegative cubes}
\label{SEC:nonnegative-cubes}

Let $A\ge 2$.
By the Selberg sieve, $\sum_{p\leq A} r_3(p) \ll A/\log{A}$.
Assuming something like \eqref{EQN:hard-HLH-clean-weight-level-d} (ideally for arbitrarily small $\delta>0$), can one prove $\sum_{p\leq A} r_3(p) \gg A/\log{A}$?
Something like \eqref{EQN:hard-HLH-clean-weight-level-d} might let one handle certain ``Type~II'' sums.
The main difficulty might instead lie in ``Type~I$_j$'' estimates (roughly corresponding to counting solutions to $dn_1n_2\cdots n_j = x^3+y^3+z^3$ for $j$ large, where $d$ is fixed).
One may be able to handle ``Type~I$_j$'' sums for $j\leq 2$ using the methods of \cite{hooley1981waring}, but it would be nice to treat larger $j$, even conditionally.

Or, assuming precise asymptotic second moments for $r_3(a)$ over $\set{a\leq A: a\equiv 0\bmod{d}}$ for $d\leq A^\delta$,
can one show that $\sum_{p\leq A} r_3(p)^2 \ll A/\log{A}$?
Note the lack of exact multiplicative structure in $d$ in the expected main term over $\set{a\leq A: a\equiv 0\bmod{d}}$.
(For $d=1$, see \cite{hooley1986some}*{Conjecture~2}.)
This may or may not be a serious obstacle.

Finally, in the conditional sense above, can one show that a positive proportion of primes $p$ have $r_3(p)\neq 0$ (i.e.~are sums of three nonnegative cubes)?
Another direction, suggested by discussion with Christian Bernert and Damaris Schindler, would be to find a sequence of arithmetic progressions $P_i$ along which $\set{a\in P_i: r_3(a)\ne 0}$ has relative density $\to 1$.

\section*{Acknowledgements}

I thank Valeriya Kovaleva, Sarah Peluse, and Peter Sarnak for inspiring me to extend my thesis work on $x^3+y^3+z^3$ from integer values to prime values.
I thank Manjul Bhargava, Tim Browning, Valeriya Kovaleva, Peter Sarnak, Katy Woo, Shuntaro Yamagishi, and Liyang Yang for conversations on closely related topics.
I thank Simona Diaconu for sharing an early draft of her enlightening senior thesis.
I thank the editors and referees for providing helpful feedback, corrections, and suggestions.
This work was partially supported by the European Union's Horizon~2020 research and innovation program under the Marie Sk\l{}odowska-Curie Grant Agreement No.~101034413.

\bibliographystyle{amsxport}
\bibliography{master.bib}

\end{document}